\documentclass[preprint,12pt]{elsarticle}

\usepackage{amssymb,color}
\usepackage{lineno,hyperref}
\modulolinenumbers[5]

\usepackage{amsmath,amsthm,amsfonts,amssymb,mathtools,mathrsfs,url}
\usepackage{paralist}
\usepackage{graphics} %% add this and next lines if pictures
\usepackage{epsfig}   %% For pictures: screened artwork should be set
\usepackage{graphicx}  \usepackage{epstopdf}
\newtheorem{theorem}{Theorem}[section]
\newtheorem{corollary}[theorem]{Corollary}
\newtheorem{lemma}[theorem]{Lemma}

\newtheorem{proposition}[theorem]{Proposition}

\newtheorem{remark}[theorem]{Remark}
\numberwithin{equation}{section}
\renewcommand{\Re}{\operatorname{Re}}

\newcommand{\BE}{\begin{equation}}
\newcommand{\EEQ}{\end{equation}}
\newcommand{\rfb}[1]{\mbox{\rm
		(\ref{#1})}\ifx\undefined\stillediting\else:\fbox{$#1$}\fi}

\newfont{\roma}{cmr10 scaled 1200}
\renewcommand{\cline}{{\mathbb C}}

\newcommand{\nline}  {{\mathbb N}}
\newcommand{\rline}  {{\mathbb R}}

\newcommand{\tline}  {{\mathbb T}}

\newcommand{\zline}  {{\mathbb Z}}

\newcommand{\dd}  {{\rm d}\hbox{\hskip 0.5pt}}
\renewcommand{\leq} {\leqslant}
\renewcommand{\geq} {\geqslant}
\newcommand{\mm}    {{\hbox{\hskip 0.5pt}}}
\newcommand{\m}     {{\hbox{\hskip 1pt}}}

\newcommand{\bluff} {{\hbox{\raise 15pt \hbox{\mm}}}}
\newcommand{\sbluff}{{\hbox{\raise 10pt \hbox{\mm}}}}
\newcommand{\Om}    {{\Omega}}

\renewcommand{\l}    {{\lambda}}

\newcommand{\FORALL} {{\hbox{$\hskip 11mm \forall \;$}}}
\newcommand{\Leloc}    {{L^2_{\rm loc}[0,\infty)}}
\newcommand{\prt}      {{\partial}}

\newcommand{\Ascr} {\mathcal{A}}

\newcommand{\Dscr} {\mathcal{D}}
\newcommand{\Hscr} {\mathcal{H}}

\newcommand{\Lscr} {\mathcal{L}}

\newcommand{\Oscr} {\mathcal{O}}
\newcommand{\Uscr} {\mathcal{U}}

\newcommand{\bbm}[1]{\left[\begin{matrix} #1 \end{matrix}\right]}
\newcommand{\sbm}[1]{\left[\begin{smallmatrix} #1
	\end{smallmatrix}\right]}
\begin{document}
	
\begin{frontmatter}
		
\title{Asymptotic behaviour of a linearized water waves \\
	   system in a rectangle\tnoteref{t1}}
		
		
\author[mymainaddress]{Pei Su \corref{mycorrespondingauthor}}
\cortext[mycorrespondingauthor]{Corresponding author}
\ead{pei.su@u-bordeaux.fr}
		
%\author[mymainaddress]{Marius Tucsnak}
		
%\ead{marius.tucsnak@u-bordeaux.fr}
%\ead[url]{https://www.math.u-bordeaux.fr/~mtucsnak/}
		
%\author[mysecondaryaddress]{George Weiss}
%\ead{                                       gweiss@tauex.tau.ac.il}
		
\address[mymainaddress]{Institut de Math\'ematiques de Bordeaux,\\
		Universit\'e de Bordeaux, \\
		351, Cours de la Lib\'eration - F 33 405 TALENCE, France}
%\address[mysecondaryaddress]{School of Electrical Engineering,
%Tel Aviv University, Ramat Aviv 69978, Israel}
\tnotetext[t1]{The author is member of the ETN network ConFlex,
   funded by the European Union's Horizon 2020 research and innovation 
   programme under the Marie Sklodowska-Curie grant agreement no. 765579.
   }

\begin{abstract}
We consider the asymptotic behaviour of small-amplitude gravity water
waves in a rectangular domain where the water depth is much smaller 
than the horizontal scale. The control acts on one lateral boundary, 
by imposing the horizontal acceleration of the water along that boundary, 
as a scalar input function $u$. The state $z$ of the system consists of 
two functions: the water level $\zeta$ along the top boundary, and its 
time derivative $\frac{\prt\zeta}{\prt t}$. We prove that the solution 
of the water waves system converges to the solution of the one dimensional 
wave equation with Neumann boundary control, when taking the shallowness 
limit. Our approach is based on a special change of variables and a 
scattering semigroup, which provide the possiblity to apply the 
Trotter-Kato approximation theorem. Moreover, we use a detailed analysis 
of Fourier series for the dimensionless version of the partial Dirichlet 
to Neumann and Neumann to Neumann operators introduced in 
\cite{Su2020stabilizability}.
\end{abstract}


\end{frontmatter}

%%%%%%%%%%++++++++++%%%%%%%%%%++++++++++%%%%%%%%%%++++++++++%%%%%%%%%%
\section{Introduction and main results} \label{intro} % Section 1
In this work we study the asymptotic behaviour of a system describing 
small-amplitude water waves in a rectangular domain, in the presence of 
a wave maker, where the horizontal scale $L$ is much larger than the 
typical water depth $h_0$. The construction of the water waves model 
begins from the so-called {\em Zakharov-Craig-Sulem formulation} (ZCS), 
which is a fully nonlinear and fully dispersive model in terms of 
the elevation of the free surface and the free surface velocity potential 
(see, for instance, Lannes' book \cite{lannes2013water}). 
Based on some assumptions on the nonlinearity and the topography of the 
fluid domain, described by the {\em shallowness parameter} 
\begin{equation}\label{mu}
\mu=\frac{h_0^2}{L^2}, 
\end{equation}
there are many asymptotic models in the shallow water regime. {\em The 
nonlinear shallow water equations} is an approximation of ZCS where all 
the terms of order $O(\mu)$ are dropped, so that it is a fully nonlinear 
and non-dispersive model. Moreover, {\em the Boussinesq equations} is an 
approximation of ZCS of order $O(\mu^2)$ with the weak nonlinearity 
assumption. The full justification (convergence) of the shallow water 
approximation of ZCS models mentioned above are provided in \cite[Chapter 
5 and Chapter 6]{lannes2013water} by considering the corresponding Cauchy 
problem in a strip domain that is unbounded in the horizontal direction. 
For more interesting asymptotic models, please refer to Lannes 
\cite{lannes2013water}, \cite{lannes2020modeling} and also thereins. 

Here, instead of considering a fluid filling an infinite strip, we consider 
the similar topic on the linearized water waves equation in a rectangular 
domain with a wave maker applied from the lateral boundary. Our aim is to 
describe the dynamics of this system when the shallowness parameter tends 
to zero. Now let us precisely state the problem.

The domain $\Om$ is bounded by a top free surface 
$\Gamma_s$ and a flat bottom $\Gamma_f$. The other two components of the 
fluid domain, denoted by $\Gamma_1$ and $\Gamma_2$, are vertical walls, 
see Figure \ref{fig1}. The fluid filling the rectangular domain 
\begin{equation*}
\Om\m=\m \left\{(x,y)\m\left|\m(x,y)\in (0,\pi L)\times (-h_0,0)\right\} 
\right. 
\end{equation*}
is assumed to be homogeneous, incompressible, inviscid and irrotational. 
There is a wave maker that acts at the left boundary of $\Om$, by 
imposing the acceleration of the fluid in the horizontal direction, 
as a scalar input signal $u$.  

%%%%%%%%%%++++++++++%%%%%%%%%%++++++++++%%%%%%%%%%++++++++++%%%%%%%%%%
\begin{figure}[htbp]
	\centering\includegraphics[width=8.5cm, height=5.5cm]{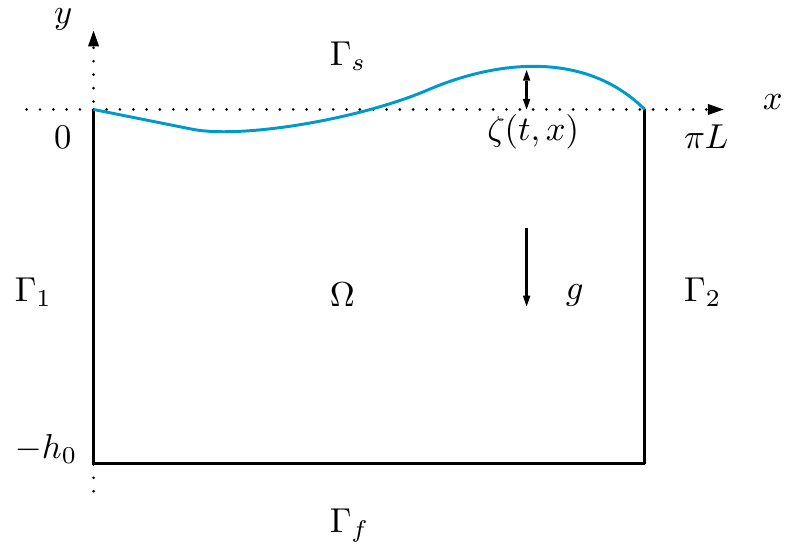}
	\caption{A rectangular domain $\Om$ filled with water} \label{fig1}
\end{figure} 
%%%%%%%%%%++++++++++%%%%%%%%%%++++++++++%%%%%%%%%%++++++++++%%%%%%%%%%

We consider the water waves system in the shallow water configurations, 
in the sense that $\mu\ll 1$. In order to study the asymptotic behaviour 
of the above system, we define the following dimensionless quantities,
\begin{equation}\label{tilde}
\tilde{x}=\frac{x}{L},\quad \tilde{y}=\frac{y}{h_0}, \quad\tilde{t}=
\frac{t}{L/\sqrt{gh_0}},\quad \tilde{\zeta}=\frac{\zeta}{a},\quad 
\tilde{\phi}=\frac{\phi}{aL\sqrt{g/h_0}}, \vspace{-0.5mm}
\end{equation}
where $a$ is the order of the surface variation, $\phi$ is the velocity 
potential of the fluid, $\zeta$ is the elevation of the top free surface 
and $g$ represents the gravity acceleration. The quantities in \rfb{tilde} 
marked with a tilde are their corresponding dimensionless version. With 
the variables $\tilde x$ and $\tilde y$, the dimensionless domain, 
denoted by $\tilde \Om$, is 
\begin{equation}\label{Om}
\tilde \Om\m=\m \left\{(\tilde x,\tilde y)\m\left|\m(\tilde x,\tilde y)
\in (0,\pi)\times (-1,0)\right\}\right.. 
\end{equation}
For the sake of simplicity, we omit the tildes in what follows and from 
now we always use the dimensionless quantities. Moreover, to avoid any 
confusion we use the notation $\zeta_\mu$ and $\phi_\mu$, instead of $\zeta$ 
and $\phi$, to represent the unknown functions in the dimensionless equation. 
The governing equations of the water waves system described above \rfb{tilde}, 
for all $t\geq 0$, are \vspace{-0.5mm}
\begin{equation} \label{gravity-1}
\left\{ \begin{aligned}
  &\Delta_\mu\m\phi_\mu(t,x,y) \m=\m 0\quad&\quad (\ (x,y)\in\Om),\\
  &\frac{\prt\zeta_\mu}{\prt t}(t,x) \m-\m \frac{1}{\mu}\m \frac{\prt\phi_\mu}
    {\prt y}(t,x,0)=0\quad&\quad (\ x \in (0,\pi)),\\
  &\frac{\prt\phi_\mu}{\prt t}(t,x,0)+\zeta_\mu(t,x) \m=\m 0\quad&\quad 
    (\ x\in(0,\pi)),\\
  &\frac{\prt\phi_\mu}{\partial x}(t,0,y) \m=\m-h(y)v(t)\quad
  &\quad (\ y \in (-1,0)),\\ &\frac{\prt\phi_\mu}{\prt y}
  (t,x,-1)=0=\m \frac{\prt\phi_\mu}{\prt x}(t,\pi,y)\quad&\quad
  (\ (x,y)\in\Om), \end{aligned} \right.
\end{equation}
where $v$ is the velocity produced by the wave maker. In the above equations 
$\Delta_\mu=\mu\frac{\prt^2}{\prt x^2}+\frac{\prt^2}{\prt y^2}$ is called 
the ''twisted'' Laplace operator (see \cite{lannes2013water}), and the function 
$h$ represents the profile of the velocity imposed by the wave maker. The system 
\rfb{gravity-1} is actually a fully linear and fully dispersive approximation 
of ZCS constrained in a rectangle.

The controllability properties of the system derived by \rfb{gravity-1}, 
as far as we know, are firstly studied in Russell and Reid \cite{reid1985boundary} 
and further in Mottelet \cite{mottelet2000controllability}. Now we recall 
here some recent works on the similar problem. Different with the control 
introduced in the system \rfb{gravity-1}, Alazard discussed in 
\cite{alazard2018stabilization} the stabilization of the nonlinear water 
waves system in a rectangle where the external pressure as the control signal 
acts on a part of the free surface, by absorbing the waves coming from the 
left. For the problem in a cubic domain, in an irregular domain and the case 
of the water waves with surface tension, please refer to Reid \cite{reid1986open} 
and \cite{reid1995control}, Craig et al. \cite{craig2012water}, Alazard et al. 
\cite{alazard2011water} and \cite{alazard2018boundary}. 
Recently, for $u\in\Leloc$, we established in our paper \cite{Su2020stabilizability} 
the well-posedness of the system \rfb{gravity-1}, and further showed that 
it can be recast as a well-posed linear control system (for this concept, 
please refer to \cite{Su2020stabilizability}, Weiss \cite{Weiss1} or Tucsnak 
and Weiss \cite{Obs_Book}).

Observe that the free surface equations of \rfb{gravity-1} determine 
the whole system, which means that if we know $\psi_\mu(t,x)=\phi_\mu(t,x,0)$, 
thereby the velocity potential $\phi_\mu$ can be obtained by solving a 
boundary value problem for Laplacian. As explained in \cite{Su2020stabilizability}, 
with the help of the Dirichlet to Neumann and the Neumann to Neumann 
operators, the system \rfb{gravity-1} reduces to a second-order evolution 
equation in terms of $\zeta_\mu$. We thus propose the corresponding initial data 
\vspace{-0.5mm}
\begin{equation}\label{initial-1}
\zeta_\mu(0,x)=\zeta_0(x),\quad ~\frac{\prt\zeta_\mu}{\prt t}(0,x)=\zeta_1(x).
\end{equation}
Therefore, we take the acceleration $u=\frac{dv}{dt}$ as the input signal.
We will provide in Section \ref{3} more details about the formulation of the governing
equations \rfb{gravity-1}. 

To state our main result, we introduce the following wave equation defined on 
$(0,\pi)$ with Neumann boundary control, i.e. for all $t\geq0$, $x\in(0,\pi)$, 
\begin{equation}\label{limit-wave}
\left\{ \begin{aligned}
&\frac{\prt^2\zeta}{\prt t^2}(t,x)-\frac{\prt^2\m\zeta}{\prt x^2}(t,x)=0,\\
&\frac{\prt\m\zeta}{\prt x}(t,0)=u(t),\quad \frac{\prt\m\zeta}{\prt x}(t,\pi)=0,\\
&\zeta(0,x)=\zeta_0(x),\quad ~\frac{\prt\zeta}{\prt t}(0,x)=\zeta_1(x).
\end{aligned} \right. \vspace{-1mm}
\end{equation}

The main contribution brought in by this work is that we justify the passage 
to the limit from the linear water waves system \rfb{gravity-1} to the system 
\rfb{limit-wave} (i.e. showing that, in an appropriate sense, $\zeta_\mu\to\zeta$) 
with the same initial data $\zeta_0$ and $\zeta_1$, as the shallowness parameter 
$\mu$ goes to zero. 

Intuitively, the rectangular domain will reduce to a one dimensional interval 
when the fluid domain becomes thinner and thinner in the vertical direction. 
From another point of view, the dispersion relation (that is the relation 
between $\omega$ and $\kappa$ when the solution takes the form  
${\rm e}^{{\rm i}(\kappa x-\omega t)}$) of the linearized water waves is 
$\omega^2=g\kappa\tanh \kappa h_0$, where $\omega$ is the angular frequency, 
$\kappa$ is the wave number and $h_0$ is the typical depth of the fluid domain. 
For more details about this, we refer to \cite[Chapter 1]{lannes2013water} 
and to Whitham \cite[Chapter 13]{whitham2011linear}). It is obvious that 
the dispersion relation is approximately $\omega^2\sim gh_0\kappa^2$ as 
$\kappa h_0\to 0$ and the phase speed $c_0=\sqrt{gh_0}$ becomes independent 
of $\kappa$. The dispersive effects drop out in this limit, and in one 
dimension, this is exactly the property of the wave equation. 

In \cite{Su2020stabilizability} we assumed that the shape function 
$h$ satisfies zero mean condition $\int_{-1}^0 h(y)\m \dd y=0$, 
to ensure the conservation of the volume of the water. This condition should 
be removed in this paper since the system under consideration is in the 
shallow water regime, where the velocity of the fluid is independent of the 
vertical variable (see, for instance, \cite{lannes2013water} and 
\cite{whitham2011linear}). In this case, $h$ should be a constant, which means 
that the velocity or the acceleration is homogeneously imposed by the wave
maker from the left edge. Without loss of generality, we might as well take 
$h=1$.

Before stating our main results, we need some notations. Let $\Oscr
\subset\rline^n$ be an open set, we use the notation $W^{k, 2}(\Oscr)$
($k\in\nline$) for the {\em Sobolev space} formed by the distributions 
$f\in\Dscr'(\Oscr)$ having the property that $\partial^\alpha f\in 
L^2(\Oscr)$ for every multi-index $\alpha\in\zline^n$ with 
$\alpha_j\geq 0$ and $|\alpha|\leq k$. For $n=1$, let $W^{s,2}(\Oscr)$
($s>0$) denote the fractional order Sobolev spaces obtained by 
interpolation via fractional powers of a positive operator (see, 
for instance, Lions and Magenes \cite{LiMa}).

Here is our main result:

\begin{theorem}\label{result}
For $u\in\Leloc$ and for any initial data $\zeta_0\in W^{1,2}[0,\pi]$
and $ \zeta_1\in L^2[0,\pi]$, let $\zeta_\mu$ be the solution of the free surface 
equations of \rfb{gravity-1} with the initial data \rfb{initial-1}, satisfying
$$ \zeta_\mu\in C([0,\infty); W^{\frac{1}{2},2}[0,\pi])
    \cap C^1([0,\infty); L^2[0,\pi]).$$
Let $\zeta$ be the weak solution of the system \rfb{limit-wave} satisfying
$$ \zeta\in C([0,\infty); W^{1,2}[0,\pi])\cap C^1([0,\infty); L^2[0,\pi]).$$
Then, for every $\tau>0$, we have \vspace{-1mm}
\begin{equation*}
  \begin{aligned}
  &\lim_{\mu\to0} \sup_{t\in[0,\tau]} \lVert \zeta_\mu-
  \zeta\rVert_{W^{\frac{1}{2},2}[0,\pi]}=0,\\
  &\lim_{\mu\to0} \sup_{t\in[0,\tau]} \left\lVert 
  \frac{\prt\zeta_\mu}{\prt t}-\frac{\prt\zeta}{\prt t} 
  \right\rVert_{L^2[0,\pi]}=0. \end{aligned} \vspace{-0.5mm}
\end{equation*}
\end{theorem}

As we expected, according to the above theorem, the elevation of the water 
waves system behaves like the displacement of a string in one dimension. 
Although we have this relationship between the water waves system and the 
wave equation, their controllability properties are much different. 
As we know, the wave equation with Neumann boundary control 
is exactly controllable (see \cite[Part III, Chapter 8]{BDDM} 
and \cite{moreira1999exact} for the sufficiently large time, and 
\cite{lasiecka1989exact} for finite time interval), while the water 
waves system \rfb{gravity-1} is even not approximately controllable 
(see \cite{reid1985boundary} and \cite{mottelet2000controllability}).

Our approach is based on the famous Trotter-Kato approximation theorem 
(see, for instance, \cite[Chapter 3]{Pazy}) and a special change of 
variables, as well as a detailed analysis of Fourier series. Moreover, 
a scattering semigroup discussed in \cite{StWe} and \cite{GrCa} 
provides the possibility for us to apply the Trotter-Kato theorem to 
control systems. 

This paper is organized as follows. We derive, in Section \ref{op2}, the 
dimensionless Dirichlet to Neumann and Neumann to Neumann operators. 
Next we do some preparations in Section \ref{3} and propose a change of 
variables to rewrite the control system, which allows us to apply the 
Trotter-Kato theorem. In Section \ref{4} we prove two important convergence 
results on the resolvent of the evolution operators. Finally, in Section 
\ref{5} we focus on the proof of the main results.

%%%%%%%%%%++++++++++%%%%%%%%%%++++++++++%%%%%%%%%%++++++++++%%%%%%%%%%
\section{Nondimensionalization of the Dirichlet to Neumann and Neumann to 
	     Neumann maps } \label{op2}

In this section we derive the dimensionless form of the Dirichlet to Neumann 
and Neumann to Neumann maps, using the dimensionless quantities introduced 
in \rfb{tilde}. We provided in \cite{Su2020stabilizability} a detailed 
construction of these two important operators allowing us to recast 
\rfb{gravity-1} as a well-posed linear control system. Following Section 4 
in \cite{Su2020stabilizability}, we go back to the definition of these two 
operators, which is closely related to two boundary value problems for 
the Laplace operator in the rectangular domain $\Om$. Note that, $\Om$ being a 
rectangle, we use separation of variables and detailed analysis of 
Fourier series to construct the dimensionless version of all related operators.

Recalling the dimensionless quantities introduced in \rfb{tilde},
it is not difficult to see that we have \vspace{-0.5mm}
\begin{equation}\label{derivative}
\frac{\prt}{\prt x}=\frac{1}{L}\frac{\prt}{\prt \tilde x},\quad 
\frac{\prt}{\prt y}=\frac{1}{h_0}\frac{\prt}{\prt \tilde y},\quad 
\frac{\prt}{\prt t}=\frac{\sqrt{gh_0}}{L}\frac{\prt}{\prt \tilde t}. \vspace{-0.5mm}
\end{equation}
Based on the above relations, we define the ''twisted'' gradient and 
Laplace operators as follows ($\mu$ is given by \rfb{mu}): \vspace{-0.5mm}
$$ \nabla_\mu=\left(\sqrt{\mu}\frac{\prt}{\prt \tilde x},\  
   \frac{\prt}{\prt \tilde y}\right), \quad \Delta_\mu=\mu\frac{\prt^2}
   {\prt \tilde x^2}+\frac{\prt^2}{\prt \tilde y^2}. \vspace{-1mm}$$ 

\begin{remark}
{\rm %Normally, we denote the differentiation with respect to time by 
%''$\cdot$'', but use ''$\partial_t$'' alternatively when deriving 
%the dimensionless quantity. In addition, 
The domain in this section is the one defined in \rfb{Om} and we still 
denote it by $\Om$ for simplicity.	}
\end{remark}
%%%%%%%%%%++++++++++%%%%%%%%%%++++++++++%%%%%%%%%%++++++++++%%%%%%%%%%
\subsection{Dirichlet and Neumann maps}

We present in this part the definition and some important remarks on 
the Dirichlet and Neumann maps in dimensionless version. Moreover, 
we state several results on the properties of these maps. The proofs of 
these results can be obtained by slight variations of the proofs of the 
corresponding results in \cite{Su2020stabilizability}, so that we omit 
the details here.

To state the definition of the Dirichlet and Neumann maps clearly, we need 
some notations. We set $H=L^2[0,\pi]$. It is known that the family 
$(\varphi_k)_{k\geq0}$ defined by \vspace{-1mm}
\begin{equation} \label{mare_definitie}
 \varphi_k=\left\{ \varphi_0=\frac{1}{\sqrt{\pi}}, \  \sqrt{\frac{2}{\pi}} 
 \cos(kx)\right\}\ \FORALL x\in [0,\pi],\vspace{-1mm}
\end{equation}
forms an orthonormal basis in $H$. The inner product in $H$ is 
denoted by $\langle\cdot, \cdot\rangle$ and the associated norm by 
$\lVert \cdot \rVert$. The Hilbert spaces $(\Hscr_\alpha)_{\alpha\geq 0}$ are 
defined by $\Hscr_0=H$ and \vspace{-1mm}
\begin{equation} \label{HALPHA}
\Hscr_\alpha \m=\m \left\{\eta\in H\ \ \left|\ \ \sum_{k\in\nline}
k^{2\alpha} |\langle\eta, \varphi_k\rangle|^2 < \infty\right\}\right. \qquad 
(\alpha\geq 0),
\end{equation}
with the inner products $\langle\eta,\psi\rangle_\alpha=
\sum_{k\in\nline}k^{2\alpha}\langle\eta, \varphi_k\rangle\m\overline{\langle\psi, 
\varphi_k\rangle}$, for all $\eta,\m\psi\in \Hscr_\alpha$.

\begin{proposition} \label{AZERO0}
With $\Om$ as in \rfb{Om}, we consider the operator
$A_1:\Dscr(A_1)\to L^2(\Om)$ defined by \vspace{-1mm}
$$ \Dscr(A_1) = \left\{f \in W^{2,2}(\Om)\ \ \left|\ \ \begin{matrix}
   f(x,0)=0,\ \frac{\partial f}{\partial y}(x,-1)=0& x\in (0,\pi)\\
   \frac{\partial f}{\partial x}(0,y)=0,\ \frac{\partial f}
   {\partial x}(\pi,y)=0 & y\in (-1,0) \end{matrix}\right.\right\},$$
$$ A_1 f = -\Delta_\mu f \FORALL f\in\Dscr(A_1).$$
Then $A_1$ is a strictly positive operator on $L^2(\Om)$.
\end{proposition}

We know from Proposition \ref{AZERO0} that the operator $A_1$ is invertible 
since it is strictly positive (see, for instance, \cite[Chapter 3]{Obs_Book}).
Moreover, we have $A_1^{-1}\in \Lscr(L^2(\Om), W^{2,2}(\Om))$. Based on this 
observation, we introduce the Dirichlet map in the following proposition.

\begin{proposition} \label{def-D}
For every $\eta\in H$, there exists a unique function $D_\mu\eta\in 
L^2(\Om)$ such that \vspace{-2mm}
\begin{equation}\label{def1}
  \int_\Om (D_\mu\eta)(x,y) \overline{g(x,y)}\m \dd x \m \dd y \m=
  -\int_0^\pi \eta(x) \overline{\frac{\partial(A_1^{-1}g)}{\partial y}
  (x,0)}\, \dd x \ \ \forall\m g\in L^2(\Om).
\end{equation}
Moreover, the operator $\eta\mapsto D_\mu\eta$ (called a {\em partial 
Dirichlet map}) is bounded from $H$ into $L^2(\Om)$.
\end{proposition}

\begin{remark}\label{Dlaplace} {\rm 
For every $\eta\in H$, we have $D_\mu\eta\in C^\infty(\Om)$ and
$\Delta_\mu(D_\mu\eta)=0$. Moreover, if $D_\mu\eta\in 
C^1(\overline\Om)$, then $D_\mu\eta$ is the unique function
in $C^2(\Om)\cap C(\overline\Om)$ that satisfies, in the classical
sense, the following dimensionless boundary value problem: 
\begin{equation} \label{LaplaceDeta}
\left\{ \begin{aligned} 
  &\Delta_\mu (D_\mu\eta)(x,y) \m=\m 0 \quad&\quad
  ((x,y)\in\Om),\\ &(D_\mu\eta)(x,0) \m=\m \eta(x),\quad \frac{\partial
  (D_\mu\eta)}{\partial y}(x,-1) \m=\m 0\quad&\quad \qquad (x\in(0,\pi)),
  \\ &\frac{\partial (D_\mu\eta)}{\partial x}(0,y) \m=\m 0, \quad \frac
  {\partial (D_\mu\eta)}{\partial x}(\pi,y) \m=\m 0\quad&\quad\qquad
  (y\in (-1,0)). \end{aligned} \right.
\end{equation}
The above equations can be obtained by taking $g=\Delta_\mu f$ in \rfb{def1} 
and cleverly choosing the test function $f\in\Dscr(A_1)$. For more details about 
this, please refer to \cite{Su2020stabilizability}.}
\end{remark}
 
\begin{lemma} \label{Deta}	
For every $\eta\in H$, $D_\mu\eta$ is given, for every $x,y\in \Om$, by
\begin{equation*}
  (D_\mu\eta)(x,y) = \sum_{k\geq0} \frac{\langle\eta,\varphi_k
  \rangle}{\cosh{(\sqrt{\mu}k)}}\varphi_k(x)\cosh{[\sqrt{\mu}k(y+1)]},
\end{equation*}
where the functions $\varphi_k$ have been defined in \rfb{mare_definitie}.
Moreover, for every $\eta\in \Hscr_3$ we have $D_\mu\eta\in C^2(\overline\Om)$.
\begin{proof}
Note that for every $\eta\in H$, we have the Fourier expansion $\eta=\sum_{k\geq0}
\langle \eta, \varphi_k\rangle \varphi_k$. Therefore, the formula of $D_\mu\eta$ 
directly follows from Proposition \ref{def-D} and Remark \ref{Dlaplace}. The 
remaining proof is similar to the corresponding one in \cite{Su2020stabilizability}.
%The formula follows from solving boundary value problem \rfb{LaplaceDeta} 
%by separation of variables.
\end{proof}
\end{lemma}

We set \vspace{-1mm}
\begin{equation*} 
  W^{1,2}_{top}(\Om) = \{f\in\Hscr^1(\Om)\ \ |\ \ f(x,0)=0,\
  x\in(0,\pi)\}.
\end{equation*}
Next we recall in what follows the definition of the Neumann map.
\begin{proposition}\label{def-N}
For every $v\in L^2[-1,0]$, there exists a unique function $N_\mu v\in
W^{1,2}_{top}(\Om)$ such that \vspace{-1mm}
\begin{equation*} 
  \int_\Om \nabla_\mu (N_\mu v)\cdot \overline{\nabla_\mu g}\, \dd x\, 
  \dd y \m=\m \int_{-1}^0 v(y) \overline{g(0,y)}\, \dd y \FORALL g\in
  W^{1,2}_{top}(\Om).
\end{equation*}
Moreover, the operator $N_\mu$, called a {\em partial Neumann map}, is
linear and bounded from $L^2[-1,0]$ to $W^{1,2}_{top}(\Om)$.
\end{proposition}

\begin{remark} \label{nicipedeparte} {\rm
The above proposition can be formulated also as follows: for every
$v\in L^2[-1,0]$, the dimensionless boundary value problem for Laplacian
\begin{equation} \label{Deltaf}
\left\{ \begin{aligned}
   & \Delta_\mu f(x,y) = 0  \quad&\quad((x,y)\in\Om),\\ & f(x,0) = 0,
   \quad \frac{\partial f}{\partial y}(x,-1) = 0\quad&\quad
   (x\in (0,\pi)),\\ & \frac{\partial f}{\partial x}(0,y) = -v,
   \quad \frac{\partial f}{\partial x}(\pi,y) = 0 \quad & \quad
   (y\in (-1,0)), \end{aligned} \right. 
\end{equation}
admits a unique weak solution $f=N_\mu v\in W^{1,2}_{top}(\Om)$. If 
$f\in C^2(\overline\Om)$ and $v\in C[-1,0]$, then $f=N_\mu v$ is 
the unique classical solution of \rfb{Deltaf}. }
\end{remark}

We note that the sequence $(\psi_k)_{k\in\nline}$ defined by
\begin{equation*} 
  \psi_k(y) = \sqrt{2}\cos{\left[(2k-1)\frac{\pi}{2}(y+1)\right]}
  \FORALL k\in\nline,\ y\in [-1,0],
\end{equation*}
is an orthonormal basis in $L^2[-1,0]$ (see \cite[Sect.~2.6]
{Obs_Book}). We define the Hilbert spaces $(\Uscr_\beta)_{\beta\geq 0}$ 
by $\Uscr_0=L^2[-1,0]$ and (for $\beta>0$) \vspace{-1mm}
\begin{equation*}
\Uscr_\beta \m=\m \left\{v\in \Uscr_0\ \ \left|\ \ \sum_{k\in
\nline} (2k-1)^{2\beta} \left|\langle v,\psi_k\rangle_{\Uscr_0}
\right|^2 < \infty\right\}\right. \m,
\end{equation*}
with the inner products given by $\langle v,w\rangle_{\Uscr_\beta} = 
\sum_{k\in\nline}(2k-1)^{2\beta} \langle v, \psi_k\rangle_{\Uscr_0}
\overline{\langle w, \psi_k\rangle_{\Uscr_0}}$, for every $v,\m w\in
\Uscr_\beta$. 

\begin{lemma} \label{Nv}
For every $v\in L^2[-1,0]$ and every $(x,y)\in\Om$ we have
\begin{equation}\label{Nvformula}
  (N_\mu v)(x,y) = \sum_{k\in\nline} a_k\cosh\left[\frac{(2k-1)}
  {2\sqrt{\mu}}\pi(x-\pi)\right]\cos{\left[(2k-1)\frac{\pi}{2}(y+1)\right]},
\end{equation}
where \vspace{-0.5mm}
$$ a_k \m=\m \frac{2\sqrt{2\mu}\langle v,\psi_k\rangle}{(2k-1)\pi
   \sinh{\left[\frac{2k-1}{2\sqrt{\mu}}\pi^2\right]}} 
   \FORALL k\in\nline. \vspace{-1mm}$$ 
Moreover, for every $v\in\Uscr_2$ we have $N_\mu v\in C^2(\overline{\Om})$.
\end{lemma}	
\begin{proof}
For every $v\in L^2[-1,0]$, $v=\sum_{k\in\nline} \langle v, \psi_k\rangle
\psi_k$. According to Proposition \ref{def-N} and Remark \ref{nicipedeparte},
using separation of variables we immediately obtain the formula \rfb{Nvformula}.
We omit the remaining details since it is almost the same with the one in 
\cite{Su2020stabilizability}. 
\end{proof}
%%%%%%%%%%++++++++++%%%%%%%%%%++++++++++%%%%%%%%%%++++++++++%%%%%%%%%%
\subsection{Dirichlet to Neumann and Neumann to Neumann maps}
In this subsection, we derive the dimensionless version of the Dirichlet 
to Neumann and Neumann to Neumann operators and study their related 
spectral properties.

\begin{corollary} \label{pe_deoarte}
Let $\gamma_1:C^1(\overline\Om)\to C[0,\pi]$ be the partial
Neumann trace operator defined by \vspace{-1mm}
$$ (\gamma_1f)(x) \m=\m \frac{\partial f}{\partial y}(x,0) \FORALL
f\in C^1(\overline\Om),\ x\in [0,\pi].$$
Then $ \tilde A_\mu$ defined by \vspace{-1mm}
$$\tilde A_\mu \eta \m=\m \gamma_1 D_\mu\eta \FORALL \eta\in \Hscr_3,$$
called a {\em partial Dirichlet to Neumann map}, is a linear bounded map from 
$\Hscr_3$ to $C[0,\pi]$.
\end{corollary}

\begin{proposition}\label{DN}
The operator $\tilde A_\mu$ introduced in Corollary {\rm
\ref{pe_deoarte}} has a unique continuous extension to an operator
$A_\mu:\Hscr_1\to H$. Then we have 
$A_\mu\varphi_k=\lambda_k\varphi_k$ with $(\lambda_0=0)$ 
\begin{equation*}
  \lambda_k\m=\m \sqrt{\mu}k\tanh (\sqrt{\mu}k) \FORALL k\in\nline,
\end{equation*}
and \vspace{-0.5mm}
\begin{equation} \label{forma1serie}
  A_\mu\eta \m=\m \sum_{k\in\nline} \l_k \langle \eta,\varphi_k
  \rangle\varphi_k \FORALL \eta\in \Hscr_1.
\end{equation}
\end{proposition}

\begin{remark}
{\rm The dimensionless Dirichlet to Neumann map $A_\mu$ introduced 
in Proposition \ref{DN} is positive, but not strictly positive, which is 
the difference with the one discussed in \cite{Su2020stabilizability}.
This is induced, as explained in the introduction, by removing zero mean 
condition from the state space, so that the system fits in the shallow 
water configurations.}
\end{remark}

\begin{corollary} \label{NN-pre}
With $\gamma_1$ as in Corollary {\rm\ref{pe_deoarte}}, define
the operator $\tilde B_\mu$ by
$$\tilde B_\mu v \m=\m \gamma_1 N_\mu v \ \FORALL v\in\Uscr_2,$$
called a {\rm partial Neumann to Neumann map},  where $N_\mu$ is the Neumann map 
introduced in Proposition {\rm \ref{def-N}}. Then $\tilde B_\mu$ is a bounded 
linear operator from $\Uscr_2$ to $C[0,\pi]$.
\end{corollary}

\begin{proposition}\label{NN}
The operator $\tilde{B}_\mu$ introduced in Corollary {\rm\ref{NN-pre}}
can be extended in a unique manner to a linear bounded operator 
$B_\mu\in\Lscr(\Uscr_0, H)$. In particular, the operator $B_\mu$ belongs to
$\Lscr(\cline, H)$ and for every $u\in\cline$ \vspace{-1mm}
\begin{equation*} 
  (B_\mu u)(x) \m=\m \sum_{k\in\nline} b_k\cosh\left[
  \frac{(2k-1)}{2\sqrt{\mu}}\pi(x-\pi)\right],
\end{equation*}
where \vspace{-0.5mm}
$$ b_k \m=\m \frac{-4u \sqrt{\mu}}{ (2k-1)\pi \sinh\left[\frac{2k-1}
	{2\sqrt{\mu}} \pi^2\right]}.$$
\end{proposition}

The proofs of Proposition \ref{DN} and Proposition \ref{NN} are completely
similar with the corresponding ones for the usual Dirichlet to Neumann and 
Neumann to Neumann maps (with dimension) in \cite{Su2020stabilizability}.
Therefore, we omit the details here. Next we introduce a convergence 
property on the Neumann to Neumann map $B_\mu$, which plays an important 
role in our arguments.

\begin{theorem}\label{Bmuconverg}
Let $B_0=-\delta_0$, where $\delta_0$ is the Dirac mass concentrated at $x=0$ 
and let $B_\mu$ be the Neumann to Neumann map defined in Proposition \ref{NN}. 
Then we have \vspace{-0.5mm}
\begin{equation}\label{maincov}
  \lim_{\mu\to0} \left\lVert\m \frac{1}{\mu} B_\mu u-B_0 u\m\right 
  \rVert_{\left(W^{1,2}[0,\pi]\right)'}= 0 \FORALL u\in\cline, \vspace{-0.5mm}
\end{equation}
where $\left(W^{1,2}[0,\pi]\right)'$ is the dual of $W^{1,2}[0,\pi]$ with 
respect to the pivot space $H$.
\end{theorem}
\begin{proof}
One readily sees that, equivalently, we need to show that for every 
$u\in\cline$ and for every $\phi\in W^{1,2}[0,\pi]$ with $\lVert \phi\rVert
_{W^{1,2}}\leq1$,\vspace{-1mm}
\begin{equation}\label{main}
  \lim_{\mu\to0}\sup_{\lVert \phi\rVert_{W^{1,2}}\leq1}\left|\left\langle 
  \frac{1}{\mu} B_\mu u\m-\m B_0 u,
  \ \phi \right\rangle\right| = 0.
\end{equation} 
According to Proposition \ref{NN}, we have \vspace{-0.5mm}
\begin{equation}\label{Bmu}
  \frac{1}{\mu} B_\mu u\m=\m \sum_{k\in\nline} c_k \cosh\left[
  \frac{(2k-1)}{2\sqrt{\mu}}\pi(x-\pi)\right],
\end{equation}
where \vspace{-1mm}
$$ c_k \m=\m \frac{-4u}{\sqrt{\mu} (2k-1)\pi \sinh\left[\frac{2k-1}
  {2\sqrt{\mu}} \pi^2\right]}. $$
We denote 
$$ f_k(x)\m=\m\sinh\left[\frac{2k-1}{2\sqrt{\mu}}\pi x\right], \qquad 
   $$
 %g(x)\m=\m\cosh\left[\frac{2k-1}{2\sqrt{\mu}}\pi x\right].
and obtain by using integration by parts that 
\begin{multline}\label{lef}
  \int_0^\pi\cosh\left[\frac{(2k-1)}{2\sqrt{\mu}}\pi(x-\pi)\right]
  \phi(x)\dd x\m=\\ \m\frac{2\sqrt{\mu}}{(2k-1)\pi}\left\{\phi(0)f_k(\pi)
  -\m\int_0^\pi\sinh\left[\frac{2k-1}{2\sqrt{\mu}}\pi(x-\pi)\right]
  \phi'(x)\dd x\right\}.
\end{multline}
Furthermore, note that
\begin{multline*}
  \int_0^\pi\sinh\left[\frac{2k-1}{2\sqrt{\mu}}\pi(x-\pi)\right]
  \phi'(x)\m\dd x=\\{\rm e}^{-\frac{2k-1}{2\sqrt{\mu}}\pi^2}\int_0^\pi 
  f_k(x)\phi'(x)\m \dd x\m-\m f_k(\pi)\int_0^\pi {\rm e}^{-\frac{2k-1}
  {2\sqrt{\mu}}\pi x}\phi'(x)\m\dd x,
\end{multline*}
we thus have the following estimate (using \rfb{lef}) 
\begin{multline}\label{I-II}
 \left|\left\langle \frac{1}{\mu} B_\mu u- B_0u,\m \phi\right\rangle\right|
 \m\leq\m \sum_{k\in\nline} \frac{32\mu |u|}{(2k-1)^4\pi^6}
 \m{\rm e}^{-\frac{2k-1}{2\sqrt{\mu}}\pi^2}\int_0^\pi \left|f_k(x)\phi'(x)
 \right|\dd x\\ 
 +\m\sum_{k\in\nline} \frac{8|u|}{(2k-1)^2\pi^2}\int_0^\pi 
 \left|{\rm e}^{-\frac{2k-1}{2\sqrt{\mu}}\pi x}\phi'(x)\right|\dd x.
\end{multline}
In the above estimate, we used the fact $\sum_{k\in\nline}\frac{1}{(2k-1)^2}=
\frac{\pi^2}{8}$ and $\sinh x\geq x^2$ for large $x$. Note that there exists 
a constant $C>0$, such that ${\rm e}^{-\frac{2k-1}{2\sqrt{\mu}}\pi^2}f_k(\pi)
\leq C$ uniformly with respect to $\mu$ and $k$, we immediately obtain that
$$\sum_{k\in\nline} \frac{\mu}{(2k-1)^4}\m{\rm e}^{-\frac{2k-1}{2\sqrt{\mu}}\pi^2}
  \int_0^\pi \left|f_k(x)\phi'(x)\right|\dd x\leq C\sum_{k\in\nline} 
  \frac{\mu}{(2k-1)^4}\m \lVert\phi'\rVert \leq C \mu.$$ 
Moreover, since $\left\lVert {\rm e}^{-\frac{2k-1}{2\sqrt{\mu}}\pi x}
\right\rVert^2\leq C\frac{2\sqrt{\mu}}{(2k-1)\pi}$ we have
$$ \sum_{k\in\nline} \frac{1}{(2k-1)^2}\int_0^\pi 
   \left|{\rm e}^{-\frac{2k-1}{2\sqrt{\mu}}\pi x}\phi'(x)\right|\dd x
  \leq C \sum_{k\in\nline} \frac{\mu^{\frac{1}{4}}}{(2k-1)^{\frac{5}{2}}}\m
  \lVert \phi'\rVert \leq C \mu^{\frac{1}{4}}.$$
Therefore, we conclude that, for every fixed $u\in\cline$, the right-hand 
side of \rfb{I-II} can be controlled by $C\mu^{\frac{1}{4}}$, which 
clearly implies \rfb{main}.
\end{proof}

%%%%%%%%%%++++++++++%%%%%%%%%%++++++++++%%%%%%%%%%++++++++++%%%%%%%%%%
\section{Operator form of the governing equations}\label{3}
In this section, we formulate the governing equations \rfb{gravity-1} as a 
well-posed LTI (linear time-invariant) system in an appropriate Hilbert space. 
To this aim, we first define a scale of Hilbert spaces associated to a certain 
operator and then derive the dimensionless control system related to \rfb{gravity-1}, 
finally formulate the control system into the one that allows us to apply the 
Trotter-Kato approximation theorem in Section \ref{5}.

For a self-adjoint positive operator $A:\Dscr(A)\to H$ with compact 
resolvents, according to the classical results (see, for instance, 
\cite[Chapter 3]{Obs_Book}), the operator $A$ is diagonalizable, also 
called Riesz-spectral operator in some literatures (for instance in 
\cite{CZ_THE_BOOK}), with an orthonormal basis $(\varphi_k)_{k\geq0}$ 
of eigenvectors and the corresponding positive eigenvalues 
$(\lambda_k)_{k\geq0}$. For any $z\in H$, we denote 
$z_k=\langle z,\varphi_k\rangle$. Moreover, we have \vspace{-1mm}
$$ \Dscr(A) \m=\m \left\{z\in H\ \left|\ \ \sum_{k\geq 0}
   (1+\left|\lambda_k\right|^2)\m |z_k|^2 < \infty\right\}\right.,
   \vspace{-1mm} $$
and $$ Az\m=\m\sum_{k\geq0}\lambda_k z_k \varphi_k\qquad (z\in\Dscr(A)).$$
For every $\alpha\in\rline$, we introduce the scale of Hilbert spaces $H_\alpha$, 
associated to the operator $A$, which is defined by ($H_0=H$) \vspace{-1mm}
$$ H_\alpha \m=\m \left\{z\in H\ \left|\ \ \sum_{k\geq 0} 
   (1+\left|\lambda_k\right|^{2\alpha})\m |z_k|^2 < \infty\right\}\right.,
   \vspace{-1mm} $$
endowed with the inner product
$$ \langle\eta,\upsilon\rangle_\alpha=\sum_{k\geq 0}(1+\left|\lambda_k
  \right|^{2\alpha})\eta_k\m\upsilon_k \FORALL \eta,\m \upsilon 
  \in H_\alpha.\vspace{-1mm}$$
It is obvious to see that, for every $\alpha\geq0$, Hilbert space $H_\alpha$ 
is actually the domain of the operator $A^\alpha$ with its graph norm $\lVert 
\cdot\rVert_{gr}$. Furthermore, for every $\alpha\in\rline$, $H_{-\alpha}$ is
the dual space of $H_\alpha$ with respect to the pivot space $H$. We will 
apply, in the following part, the above definition of a scale of Hilbert 
spaces to different operators.

%\begin{remark}
%{\rm Since the dual Hilbert spaces $\Dscr(A^{-\alpha})$ $(\alpha>0)$ are larger than
%$H$, we need to describe them in larger Hilbert spaces rather than in $H$. For instance, 
%for $\alpha=\frac{1}{2}$, the space $\Dscr(A^{-1/2})$ discussed in 
%\cite[Proposition 3.4.8]{Obs_Book} is considered in $\Dscr(A^{-1})$. }
%\end{remark}

Next we formulate the equations \rfb{gravity-1} into a second-order evolution 
equation in terms of $\zeta_\mu$. Recalling the definition of the 
Dirichlet to Neumann map $A_\mu$ and the Neumann to Neumann map $B_\mu$ 
in Section \ref{op2}, we immediately obtain from the structure of the governing equations 
\rfb{gravity-1} that \vspace{-1mm}
\begin{equation}\label{relation}
\frac{\prt \phi_\mu}{\prt  y}( t, x,0)=A_\mu\psi_\mu( t, x)+B_\mu v( t),\vspace{-1mm}
\end{equation}
where $t\geq0$, $x\in(0,\pi)$ and  $\psi_\mu( t, x)= \phi_\mu(t,x,0)$. Taking the 
derivative of the second equation in \rfb{gravity-1} with respect to time and 
eliminating $\psi_\mu(t,x)$ by using the third equation of \rfb{gravity-1}, we get 
the second-order control system associated to \rfb{gravity-1}, i.e. for all $t\geq0$, 
$x\in(0,\pi)$,
\begin{equation}\label{Amu}
  \left\{ \begin{aligned} 
  &\frac{\prt^2\zeta_\mu}{\prt t^2}(t,x)\m +\m \frac{1}{\mu}A_\mu 
  \zeta_\mu(t,x)=\frac{1}{\mu}B_\mu u(t),\\
  &\zeta_\mu(0,x)=\zeta_0(x),\quad \frac{\prt \zeta_\mu}{\prt t}(0,x)
  =\zeta_1(x), \end{aligned}\right. 
\end{equation}
where $u=\frac{dv}{dt}$ is the input signal, the operators $A_\mu$ and $B_\mu$ 
are defined in Proposition \ref{DN} and Proposition \ref{NN}, respectively. 
Moreover, we introduce the operator $A_0:\Dscr(A_0)\to H$ as follows: \vspace{-1mm}
\begin{equation}\label{def-A0}
A_0=-\frac{d^2}{dx^2}\qquad \Dscr(A_0)=\left\{f\in W^{2,2}[0,\pi]\m 
\left|\m  \frac{df}{dx}(0)=\frac{df}{dx}(\pi)=0\right.\right\}. 
\end{equation}
With the operators $B_0$ defined in Theorem \ref{Bmuconverg}, we consider the 
following evolution equation, i.e. for all $t\geq0$, $x\in(0,\pi)$, \vspace{-1mm}
\begin{equation}\label{A0}
  \left\{ \begin{aligned} 
  &\frac{\prt^2\zeta}{\prt t^2}(t,x)\m +\m A_0 \zeta(t,x)=B_0 u(t),\\
  &\zeta(0,x)=\zeta_0(x), \quad \frac{\prt \zeta}{\prt t}(0,x)=\zeta_1(x). 
  \end{aligned}\right.\vspace{-1mm}
\end{equation}

It is known that the operator $A_0$ defined in \rfb{def-A0} is diagonalizable 
with the eigenvalues $k^2$ and the corresponding eigenvectors $\varphi_k$ are
given in \rfb{mare_definitie}. For the operator $A_0$, we denote by 
$\mathbb{H}_\alpha$ with $\alpha\in\rline$ the scale of Hilbert spaces which has 
been introduced at the beginning of this section. Notice that the Dirichlet to 
Neumann operator $A_\mu$ in Proposition \ref{DN} is also diagonalizable, so that, 
for $\alpha\in\rline$, we denote by $H_{\mu,\alpha}$ the scale of Hilbert spaces 
associated to the operator $\frac{1}{\mu}A_\mu$. Therefore, we have $\mathbb{H}_0
=\mathbb{H}=H=H_{\mu,0}$ and $\mathbb{H}_{-\alpha}$ (or $H_{\mu,-\alpha}$) is the 
dual space of $\mathbb{H}_\alpha$ (or $H_{\mu,\alpha}$) with respect to the pivot 
space $H$. It is not difficult to see that actually we have $\mathbb{H}_{\frac{1}{2}}
=W^{1,2}[0,\pi]$. For more details on a scale of Hilbert space, please refer to 
\cite[Chapter 2]{Obs_Book}. 

We mention that the operator $B_\mu$ is bounded (see Proposition \ref{NN}), i.e. 
$B_\mu\in \Lscr(\mathbb{C},H)$, and $B_0$ induces an admissible control operator 
in the first-order system associated to \rfb{A0}  with the state $\sbm{\zeta \\ 
\frac{\prt\zeta}{\prt t}}$ (please refer to \cite[Proposition 6.2.5]{Obs_Book}), 
although it is unbounded (not contained in the state space), i.e. $B_0\in\Lscr
(\mathbb{C},\mathbb{H}_{-\frac{1}{2}})$. Based on the above analysis, the system 
\rfb{A0} is well-defined.

\begin{remark}
{\rm According to Proposition \ref{DN}, the eigenvalues of $\frac{1}{\mu}A_\mu$ 
are $\frac{k\tanh (\sqrt{\mu}k)}{\sqrt{\mu}}$, which is equivalent to $k$ for fixed 
$\mu\in(0,1)$. Therefore, for every $\alpha\geq0$, Hilbert space $H_{\mu,\alpha}$ 
is actually equivalent to $\Hscr_\alpha$ introduced in \rfb{HALPHA}. Moreover, 
according to interpolation theory (see, for instance, \cite{LiMa}, \cite[Part II]{BDDM} 
and \cite{chandler2015interpolation}), for $\alpha\in(0,1)$, the scale of Hilbert space 
$\Hscr_\alpha$ is exactly the classical Sobolev space $W^{\alpha,2}[0,\pi]$.}
\end{remark}

\begin{remark}
{\rm The initial boundary value problem \rfb{limit-wave} is a well-posed 
boundary control system (for this concept, see for instance \cite[Chapter 10]
{Obs_Book}), which is equivalent to \rfb{A0} in weak sense, that is, for every
$u\in\Leloc$, for every $\zeta_0\in\mathbb{H}_{\frac{1}{2}} $ and 
$\zeta_1\in\mathbb{H}$, there exists a unique function \vspace{-1mm}
$$ \zeta\in C([0,\infty); \mathbb{H}_{\frac{1}{2}})\cap C^1([0,\infty);
   \mathbb{H}), \vspace{-0.5mm}$$
such that $\zeta(0,x)=\zeta_0$ and it satisfies, for every $t\geq0$ and 
every $\psi\in \mathbb{H}_{\frac{1}{2}}$, \vspace{-1mm}
\begin{multline*}
  \int_0^\pi \frac{\prt\zeta}{\prt t}(t,x)\overline{\psi(x)}\m\dd x -\int_0^\pi
  \zeta_1(x) \overline{\psi(x)}\m\dd x\\
  \m=\m -\int_0^t\int_0^\pi \frac{\prt\zeta}{\prt x}(\sigma,x)\overline{
  \frac{d\psi}{d x}(x)} \m\dd x\,\dd \sigma - \int_0^t u(\sigma) 
  \overline{\psi(0)} \m\dd \sigma.
\end{multline*}  }
\end{remark}

In what follows, we are ready to study the asymptotic behaviour of 
the system \rfb{Amu} when $\mu$ goes to zero. We shall consider the 
relationship between the solutions of \rfb{Amu} and \rfb{A0}. Normally, 
the state of the control system is taken as $\sbm{\zeta \\ 
\frac{\prt\zeta}{\prt t}}$, but the main problem lies in the difference 
of the energy space of \rfb{Amu} and \rfb{A0}, one is $H_{\mu,\frac{1}{2}}\times H$ 
\vspace{+0.5mm} and the other is $\mathbb{H}_{\frac{1}{2}}\times\mathbb{H}$. 
It means that we cannot apply the Trotter-Kato theorem directly. According 
to the classical semigroup theory (see, for instance, \cite{Obs_Book} and 
\cite{CZ_THE_BOOK}), for $\zeta_0\in\mathbb{H}_{\frac{1}{2}}$ and 
$\zeta_1\in\mathbb{H}$, \rfb{Amu} and \rfb{A0} admit a unique solution 
$\zeta_\mu$ and $\zeta$, respectively, which satisfy 
$$ \zeta_\mu\in C([0,\infty);H_{\mu,\frac{1}{2}})\cap C^1([0,\infty);H),$$
and \vspace{-1mm}
$$ \zeta\in C([0,\infty); \mathbb{H}_{\frac{1}{2}})\cap C^1([0,\infty);
   \mathbb{H}).  $$ 
We thus consider the following change of variables,
\begin{equation}\label{alphamu}
  \alpha_\mu:=\frac{\prt\zeta_\mu}{\prt t},\qquad 
  \beta_\mu:=\left(\frac{1}{\mu}A_\mu\right)^{1/2}\zeta_\mu,
  \vspace{-1mm}
\end{equation}
and 
\begin{equation}\label{alpha}
\alpha:=\frac{\prt\zeta}{\prt t},\qquad \beta:=A_0^{1/2}\zeta,
\end{equation} 
where $A_0$ and $A_\mu$ are introduced in \rfb{def-A0} and Proposition 
\ref{DN}, respectively. In this way, we have $\alpha_\mu,\m \beta_\mu\in 
C([0,\infty);H)$ and $\alpha,\m \beta\in C([0,\infty);\mathbb{H})$. 
Setting 
$$ w_\mu(t)= \bbm{\alpha_\mu(t,\cdot)\\ \beta_\mu(t,\cdot)} \quad 
   \text{and} \quad w(t)= \bbm{\alpha(t,\cdot)\\ \beta(t,\cdot)}, $$
we obtain from \rfb{Amu} and \rfb{A0} that
\begin{equation}\label{newmu-1st}
  \left\{ \begin{aligned}
  &\frac{d w_\mu}{d t}(t)=\mathscr{A}_\mu w_\mu(t)+
  \mathscr{B}_\mu u(t),\\
  & w_\mu(0)=w_{\mu,0},\end{aligned} \right.
\end{equation}
and 
\begin{equation}\label{new-1st}
\left\{ \begin{aligned}
&\frac{d w}{d t}(t)=\mathscr{A}_0 w(t)+\mathscr{B}_0 u(t),\\
& w(0)=w_0, \end{aligned} \right.
\end{equation}
where 
\begin{equation}\label{curA}
\mathscr{A}_\mu\m=\m\bbm{ 0 & -\left(\frac{1}{\mu}A_\mu\right)
	^{1/2} \\ \left(\frac{1}{\mu}A_\mu\right)^{1/2} 
	& 0},\qquad \mathscr{A}_0\m=\m\bbm{ 0 & -A_0^{1/2} \\ 
	A_0^{1/2} & 0}, 
\end{equation}
\begin{equation}\label{curB}
\mathscr{B}_\mu=\bbm{\frac{1}{\mu}B_\mu \\ 0}, \qquad 
\mathscr{B}_0=\bbm{B_0 \\ 0},
\end{equation}
and 
\begin{equation}\label{initial}
w_{\mu,0}=\bbm{\zeta_1\\ \left(\frac{1}{\mu}A_\mu\right)
	^{1/2}\zeta_0},\qquad w_0=\bbm{\zeta_1\\
	A_0^{1/2}\zeta_0}.
\end{equation}

Let $X=H\times H$, then the operator $\mathscr{A}_\mu: 
\Dscr(\mathscr{A}_\mu)\to X$ with $\Dscr(\mathscr{A}_\mu)=H_{\mu,\frac{1}{2}}
\times H_{\mu,\frac{1}{2}}$ and $\mathscr{A}_0: \Dscr(\mathscr{A}_0)\to X$ 
with $\Dscr(\mathscr{A}_0)=\mathbb{H}_{\frac{1}{2}}\times\mathbb{H}_
{\frac{1}{2}}$. Furthermore, it is not difficult to see that 
$\mathscr{B}_\mu\in\Lscr(\cline,X)$ and $\mathscr{B}_0\in \Lscr(\cline, 
\mathbb{H}_{-\frac{1}{2}}\times\mathbb{H}_{-\frac{1}{2}})$. With the help 
of the new variables defined in \rfb{alphamu} and \rfb{alpha}, the control 
systems we are now focusing on are \rfb{newmu-1st} and \rfb{new-1st}, 
which possess the same state space $X$ and provide the possibility to 
apply the Trotter-Kato theorem.

Based on the structure of the operators $\mathscr{A}_\mu$ and 
$\mathscr{A}_0$, we introduce the following lemma, which is probably known 
in the semigroup community. However, for the sake of completeness (and with 
no claim of originality) we give here its precise statement and a short proof. 
For simplicity, we denote by $R(\lambda:A)=(\lambda I-A)^{-1}$ the resolvent 
of $A$ with $\lambda\in \rho(A)$ (resolvent set of $A$).

\begin{lemma}\label{skew-adjoint}
Let the operator $A:\Dscr(A)\to H$ be positive (i.e. $A\geq 0$) with compact 
resolvents. Then the operator $\mathscr{A}:\Dscr(\mathscr{A})\to X$ defined by 
$$ \Dscr(\mathscr{A})=\Dscr\left(A^{1/2}\right)\times\Dscr\left(A^{1/2}\right),$$
$$  \mathscr{A}\bbm{\varphi \\ \psi}\m=\m\bbm{-A^{1/2}\psi \\ A^{1/2}\varphi }, 
   \FORALL \bbm{\varphi \\ \psi}\in \Dscr(\mathscr{A}), $$
generates a unitary group on $X$.
\begin{proof}
The operator $\mathscr{A}$ is obviously skew-symmetric since 
$\Re\langle \mathscr{A}w, w\rangle_X=0$ for all $w=\sbm{w_1 \\ w_2}\in 
\Dscr(\mathscr{A})$. Note that, for every $\sbm{f \\ g}\in X$, there exists 
$\varphi$ and $\psi$ defined by
\begin{equation*}
\varphi=-R(-1:A)(f+A^{1/2}g),\qquad \psi=R(-1:A)(-g+A^{1/2}f),
\end{equation*}
satisfy $\varphi,\m\psi\in\Dscr(A^{1/2})$ and
\begin{equation}\label{reso-l}
  \left(I+\mathscr{A}\right)\bbm{\varphi \\ \psi}=\bbm{f \\ g}.
\end{equation}
Indeed, note that since $A$ is positive, then $\sigma(A) \subset[0,\infty)$, 
which implies $-1\in\rho(A)$, so that the operator $I+A$ is invertible. Next we 
show that $\varphi,\m\psi\in\Dscr(A^{1/2})$. The positive operator $A^{1/2}: 
H_{\frac{1}{2}}\to H$ has a unique extension (still denoted by $A^{1/2}$) 
such that $A^{1/2}\in\Lscr(H, H_{-\frac{1}{2}})$, where the Hilbert spaces 
$H_s$ $(s\in\rline)$ is the scale of Hilbert space associated to the operator 
$A$. Moreover, $A: H_1\to H$ also has a unique extension such that 
$A\in\Lscr(H_{\frac{1}{2}},H_{-\frac{1}{2}})$, which implies that 
$R(-1,A)\in \Lscr(H_{-\frac{1}{2}}, H_{\frac{1}{2}})$. Thus, for every $g \in H$, 
$R(-1:A)A^{1/2}g\in H_{\frac{1}{2}}$. Since $A$ is positive with compact 
resolvents we obtain that $A$ is diagonalizable. According to the properties 
of diagonablizable operator (see, for instance, \cite[Section 3.6]{Obs_Book}), 
it is straight to verify that $R(-1,A)$ commutes with $A^{1/2}$, i.e. 
$R(-1,A)A^{1/2}f=A^{1/2}R(-1,A)f$, for every $f\in H$. Therefore, $\varphi$ 
and $\psi$ defined in the above formally satisfy \rfb{reso-l}. It follows 
that $I+\mathscr{A}$ is onto. 
		
Similarly, we get $I-\mathscr{A}$ is also onto. Then $\mathscr{A}$ 
is skew-adjoint on $X$ (see, for instance, \cite[Proposition 3.7.3]{Obs_Book}), 
so that, according to Stone's theorem, $\mathscr{A}$ generates a group of unitary 
operators on $X$. 
\end{proof}
\end{lemma}

\begin{remark}\label{embed}
{\rm The scale of Hilbert spaces $H_s$ ($s\in\mathbb{R}$) associated to 
the positive operator $A$, where $H_{-\alpha}$ is the dual of $H_\alpha$ 
($\alpha\geq0$) with the pivot space $H$, have the dense and 
continuous embeddings (see \cite[Section 3.4]{Obs_Book})\vspace{-1mm}
$$ H_1\subset H_{\frac{1}{2}}\subset H\subset H_{-\frac{1}{2}}
\subset H_{-1}. \vspace{-1mm}$$
The operators $A^{1/2}$ and $A$ have unique extensions such that 
$\tilde{A}^{1/2}\in\Lscr(H,H_{-\frac{1}{2}})$ and $\tilde{A}\in
\Lscr(H,H_{-1})$. Moreover, according to \cite[Section 2.10]{Obs_Book}, 
for every $\lambda\in\rho(A)$, the resolvent also have the corresponding 
extensions $R(\lambda:\tilde{A}) \in \Lscr(H_{-1}, H)$ and 
$R(\lambda:\tilde{A}^{1/2}) \in\Lscr(H_{-\frac{1}{2}}, H)$, 
which are unitary.  }
\end{remark}

%%%%%%%%%%++++++++++%%%%%%%%%%++++++++++%%%%%%%%%%++++++++++%%%%%%%%%%
\section{Operator convergence} \label{4}
This section contains the main ingredients of the convergence results, based on 
appropriate decomposition of the Fourier series describing the operators 
$\mathscr{A}_\mu$, $\mathscr{A}_0$ introduced in \rfb{curA} and the control operators 
$\mathscr{B}_\mu$, $\mathscr{B}_0$ in \rfb{curB}, which play an important 
role in the proof of Theorem \ref{result}.

We use the notation $A\in G(M,\omega)$ in what follows for an operator $A$, 
which is the generator of a $C_0$-semigroup $\tline(t)$ satisfying 
$\lVert\tline(t)\rVert\leq M{\rm e}^{\omega t}$ for every $t\geq0$. 
With this notation, Lemma \ref{skew-adjoint} implies that $\mathscr{A}_\mu,\m 
\mathscr{A}_0\m\in G(1,0)$ since $\frac{1}{\mu}A_\mu$ and $A_0$ are positive. 

\begin{lemma}\label{u=0}
With the operators $\mathscr{A}_\mu$ and $\mathscr{A}_0$ defined 
in \rfb{curA}, for every $\sbm{f \\ g}\in X$ we have\vspace{-0.5mm}
\begin{equation}\label{conv-resol}
  \lim_{\mu\to0}R(1:\mathscr{A}_\mu)\bbm{f \\ g}=
  R(1:\mathscr{A}_0)\bbm{f \\ g}\qquad 
  {\rm in} \quad X.
\end{equation}
\end{lemma}
\begin{proof}
According to Lemma \ref{skew-adjoint}, the operators $\mathscr{A}_\mu, 
\mathscr{A}_0\m\in G(1,0)$, which implies that $1\in \rho(\mathscr{A}_\mu)
\cap\rho(\mathscr{A}_0)$. We denote, for every $\sbm{f \\ g}\in X$, 
$$ \bbm{\varphi_\mu \\ \psi_\mu}\m=\m R(1:\mathscr{A}_\mu)
   \bbm{f \\ g},\qquad \bbm{\varphi_0 \\ \psi_0}=R(1:\mathscr{A}_0)
   \bbm{f \\ g}. $$
It follows that \vspace{-0.5mm}
$$\varphi_\mu=-R\left(-1:\frac{1}{\mu}A_\mu\right)
  \left[\left(\frac{1}{\mu}A_\mu\right)^{1/2}f+g\right], \vspace{-1mm}$$
$$\psi_\mu=R\left(-1:\frac{1}{\mu}A_\mu\right)
  \left[-f+\left(\frac{1}{\mu}A_\mu\right)^{1/2}g\right],\vspace{-1mm}$$
and \vspace{-1mm}
$$ \varphi_0=-R(-1:A_0)(A_0^{1/2}f+g),\qquad \psi_0=R(-1:A_0)
   (-f+A_0^{1/2}g).$$ 
As explained in the proof of Lemma \ref{skew-adjoint} and Remark \ref{embed}, 
we have $\sbm{\varphi_\mu \\ \psi_\mu}\in X$ and 
$\sbm{\varphi_0 \\ \psi_0}\in X$, which means that the expression in 
\rfb{conv-resol} makes sense.

Next we prove the convergence of each component of \rfb{conv-resol} in 
$H$ as $\mu$ goes to zero. Since $\frac{1}{\mu}A_\mu$ and $A_0$ are 
diagonablizable operators, according to \cite[Proposition 2.6.2]{Obs_Book},
we obtain from \rfb{forma1serie} and \rfb{def-A0} that \vspace{-0.5mm}
$$ R\left(-1:\frac{1}{\mu}A_\mu\right)g-R(-1:A_0)g\m
  =\m\sum_{k\in\nline}F_\mu(k)\langle g,\varphi_k\rangle \varphi_k, 
  \vspace{-0.5mm} $$
with \vspace{-0.5mm}
\begin{equation}\label{F}
  F_\mu(k) =\m\frac{1}{1+k^2}-\frac{1}{1+\frac{k\tanh(\sqrt{\mu}k)}
  {\sqrt{\mu}}}.
\end{equation} 
Denoting $h(x)=\frac{\tanh x}{x}$ ($h(0):=1$), we have \vspace{-0.5mm}
$$ F_\mu(k)=-\int_0^{\sqrt{\mu}k} \left(\frac{1}{1+k^2h(x)}\right)'
   \dd x. \vspace{-0.5mm}$$
Note that 
$$ \left|\left(\frac{1}{1+k^2h}\right)'\right| \m\leq\m 
   \frac{-h'}{k^2 h^2},$$
and $\frac{-h'}{h^2}\leq 1$ on $[0,\infty)$, which implies that 
$|F_\mu(k)|\leq \frac{\sqrt{\mu}}{k}$. We thus arrive at 
\begin{equation}\label{lemma-1}
  \left\lVert R\left(-1:\frac{1}{\mu}A_\mu\right)g-R(-1:A_0)g\m
  \right\rVert^2\leq  \mu\lVert g\rVert^2.
\end{equation}
Similarly, for every $f\in H$ we have \vspace{-0.5mm}
$$ R\left(-1:\frac{1}{\mu}A_\mu\right)\left(\frac{1}{\mu}A_\mu
   \right)^{1/2}f -R(-1:A_0) A_0^{1/2}f\m=\m
   \sum_{k\in\nline}G_\mu(k) \langle f,\varphi_k\rangle \varphi_k,
   \vspace{-0.5mm} $$
with  
\begin{equation}\label{Gmu}
  G_\mu(k)=\frac{k}{1+k^2}-\frac{\left(\frac{k\tanh(\sqrt{\mu}k)}
  {\sqrt{\mu}}\right)^{1/2}}{1+\frac{k\tanh(\sqrt{\mu}k)}{\sqrt{\mu}}}.
\end{equation}
We similarly have \vspace{-1mm}
$$ G_\mu(k)=-\int_0^{\sqrt{\mu}k} \left(\frac{k h^{1/2}(x)}{1+k^2h(x)}
   \right)' \dd x.  $$
It is not difficult to see that 
$$\left|\left(\frac{k h^{1/2}}
{1+k^2h}\right)' \right|\leq \frac{-h'}{2k\m h^{3/2}}.$$
Moreover, note that $\frac{-h'}{h^{3/2}}$ is bounded on $[0,\infty)$, 
we thus obtain that 
\begin{equation}\label{Gmu1}
|G_\mu(k)|\leq C \sqrt{\mu},
\end{equation}
which yields that
$$ \left\lVert R\left(-1:\frac{1}{\mu}A_\mu\right)\left(\frac{1}{\mu}A_\mu
   \right)^{1/2}f -R(-1:A_0) A_0^{1/2}f\m \right\rVert^2 
   \leq C \mu\left\lVert f\right\rVert^2. $$ 
This, together with \rfb{lemma-1}, implies that $\displaystyle\lim_{\mu\to0}
\varphi_\mu=\varphi_0$ in $H$ and $\displaystyle\lim_{\mu\to0}\psi_\mu=\psi_0$ 
in $H$, which ends the proof.
\end{proof}

\begin{lemma}\label{A-B}
Let $\mathscr{A}_\mu$ and $\mathscr{A}_0$ be the same as in Lemma
\ref{u=0}. For the operators $\mathscr{B}_\mu$ and $\mathscr{B}_0$ 
defined in \rfb{curB}, for every $ u\in\cline$ we have
\begin{equation*}
  \lim_{\mu\to0}R(1:\mathscr{A}_\mu)\mathscr{B}_\mu u= R(1:\mathscr{A}_0)
  \mathscr{B}_0 u \qquad {\rm in} \quad X.
\end{equation*}
\end{lemma}

\begin{proof}
For every $u\in\cline$, it is obvious that $R(1:\mathscr{A}_\mu)
\mathscr{B}_\mu u \in X$ since $\mathscr{B}_\mu\in\Lscr(\cline,X)$. 
Note that $\mathscr{B}_0\in \Lscr(\cline, \mathbb{H}_{-\frac{1}{2}}
\times\mathbb{H}_{-\frac{1}{2}})$, according to Remark \ref{embed}, 
then we have $R(1:\mathscr{A}_0)\mathscr{B}_0 u\in X$. For every 
$u\in\cline$, let \vspace{-1mm}
$$ \bbm{\tilde \varphi_\mu \\ \tilde \psi_\mu}\m=\m R(1:\mathscr{A}_\mu)
   \mathscr{B}_\mu u,\qquad \bbm{\tilde \varphi_0 \\ \tilde \psi_0}=
   R(1:\mathscr{A}_0)\mathscr{B}_0 u. $$
We immediately have \vspace{-1mm}
$$ \tilde \varphi_\mu=-R\left(-1:\frac{1}{\mu}A_\mu\right) 
   \frac{1}{\mu}B_\mu u, \vspace{-1mm}$$
$$ \tilde \psi_\mu=-R\left(-1:\frac{1}{\mu}A_\mu\right)\left(
   \frac{1}{\mu}A_\mu \right)^{1/2}\frac{1}{\mu}B_\mu u,$$
and 
$$ \tilde \varphi_0=-R(-1:A_0)B_0u,\qquad \tilde \psi_0
   =-R(-1:A_0)A_0^{1/2}B_0u.$$
For the sake of clarity, we split the remaining proof into two steps.

{\bf Step 1}: {We prove that $\displaystyle\lim_{\mu\to0}\tilde \varphi_\mu 
=\tilde \varphi_0$ in $H$.} 
To this aim, we first note, using a triangle inequality, that \vspace{-1mm}
\begin{multline}\label{I}
  \left\lVert \tilde \varphi_\mu-\tilde \varphi_0\right\rVert
  \leq \m\left\lVert \left[R\left(-1:\frac{1}{\mu}A_\mu\right)-
  R(-1:A_0)\right]\frac{1}{\mu}B_\mu u\m\right\rVert\\
  \m+\m\left\lVert R(-1:A_0)\left(\frac{1}{\mu}B_\mu u
  -B_0u\right)\right\rVert, \vspace{-2mm}
\end{multline}
where we used the fact that $R(-1:A_0)\frac{1}{\mu}B_\mu u \in H$. 
We note that $R(-1:A_0)\in\Lscr(\mathbb{H}_{-\frac{1}{2}}, 
\mathbb{H}_{\frac{1}{2}})$ and this is unitary. Indeed, by the 
Riesz representation theorem, for every $f\in \mathbb{H}_{-\frac{1}{2}}$, 
there exists a unique $\varphi\in \mathbb{H}_{\frac{1}{2}}$ 
such that \vspace{-1mm}
\begin{equation}\label{lax}
  \langle \varphi, \psi\rangle +\langle A_0^{1/2}\varphi, A_0^{1/2}
  \psi\rangle = \langle f, \psi\rangle_{\mathbb{H}_{-\frac{1}{2}}, 
  \mathbb{H}_{\frac{1}{2}}}
  \qquad \FORALL \psi\in\mathbb{H}_{\frac{1}{2}}, 
\end{equation}
which is $(I+A_0)\varphi=f$ in $ \mathbb{H}_{-\frac{1}{2}}$. 
Taking $\psi=\varphi$ in \rfb{lax}, it follows that 
$$ \lVert\m\varphi\m\rVert_{\mathbb{H}_{\frac{1}{2}}}^2\leq \lVert f\m
   \rVert_{\mathbb{H}_{-\frac{1}{2}}}\cdot\lVert\m\varphi\m\rVert_{
   \mathbb{H}_{\frac{1}{2}}}, $$
which implies that $\lVert \varphi\rVert_{\mathbb{H}_{\frac{1}{2}}}\leq 
\lVert f\rVert_{\mathbb{H}_{-\frac{1}{2}}}$, i.e. $\lVert R(-1:A_0)
\rVert_{\Lscr(\mathbb{H}_{-\frac{1}{2}},\mathbb{H}_{\frac{1}{2}})}\leq 1$. 
This, together with \rfb{maincov}, yields that the last term 
on the right-hand side of \rfb{I} converges to zero, where we used the fact 
that $\mathbb{H}_{\frac{1}{2}}=W^{1,2}[0,\pi]$ and the continuous dense 
embedding $\mathbb{H}_{\frac{1}{2}}\hookrightarrow H$. 

Next we estimate the square of the first norm on the right side of \rfb{I}. 
According to \cite[Proposition 2.6.2]{Obs_Book}, we write it in form of 
Fourier series, which reads 
\begin{equation}\label{I1}
\sum_{k\in\nline} \left|F_\mu(k)\right|^2\cdot\left|
\left\langle \frac{1}{\mu}B_\mu u, \varphi_k\right\rangle\right|^2, 
\end{equation} 
where $F_\mu(k)$ has been defined in \rfb{F}. According to \rfb{Bmu}, 
it is not difficult to see that 
\begin{equation}\label{Bmuformula}
\left\langle \frac{1}{\mu}B_\mu u, \varphi_k\right\rangle \m =\m
\frac{-2\sqrt{2}u}{ \mu\sqrt{\pi}} \sum_{l\in\nline}\m H_\mu(k,l),
\end{equation}
with 
$$ H_\mu(k,l)\m=\m \frac{1}{\left(\frac{2l-1}{2\sqrt{\mu}}\pi\right)^2+k^2}.$$
%In the above calculation, we used the basic integration 
%$$ \int{\rm e}^{ax}\cos(bx) \dd x\m=\m \frac{{\rm e}^{ax}}{a^2+b^2}
   %\left[a\cos(bx)+ b\sin(bx)\right]+C.$$
Moreover, we readily see that 
\begin{equation}\label{Hmu}
 \sum_{l\in\nline}H_\mu(k,l) \leq \sum_{l\in\nline}\frac{4\mu}{(2l-1)^2\pi^2}
 =\frac{\mu}{2}.
\end{equation}
Recalling that $|F_\mu(k)|\leq \frac{\sqrt{\mu}}{k}$, we thus 
conclude that \rfb{I1} can be controlled by $C\mu$. Therefore, we obtain that 
\rfb{I} converges to zero as $\mu\to0$.

{\bf Step 2}: {We prove that $\displaystyle\lim_{\mu\to0}\tilde \psi_\mu =\tilde 
	\psi_0$ in $H$.}  
Since $R(-1:A_0)A_0^{1/2}\frac{1}{\mu}B_\mu u\in H$ for fixed $\mu$, we consider 
the following triangle inequality: \vspace{-1mm}
\begin{multline}\label{II}
  \left\lVert \tilde \psi_\mu-\tilde \psi_0 \right\rVert
  \leq \m\left\lVert R(-1:A_0)A_0^{1/2}\left(\frac{1}{\mu}B_\mu u
  -B_0u\right)\right\rVert \\
  +\m\left\lVert \left[R\left(-1:\frac{1}{\mu}A_\mu\right) \left(
  \frac{1}{\mu}A_\mu \right)^{1/2}-R(-1:A_0)A_0^{1/2}\right]
  \frac{1}{\mu}B_\mu u\m\right\rVert.\vspace{-2mm}
\end{multline}
We first, using Fourier series, prove that $R(-1:A_0)A_0^{1/2}\in 
\Lscr(\mathbb{H}_{-\frac{1}{2}}, H)$. For every $f\in 
\mathbb{H}_{-\frac{1}{2}}$, we have $f=\langle f, 1\rangle\frac{1}{\pi}
+\sum_{k\in\nline} (1+k^2)^{-1/2} \langle f,\varphi_k\rangle\varphi_k$ 
and then, according to \cite[Proposition 2.6.2]{Obs_Book}, \vspace{-1mm}
$$ R(-1:A_0)A_0^{1/2}f=\sum_{k\in\nline} \frac{-k}{(1+k^2)^{3/2}}\langle f,
  \varphi_k\rangle\varphi_k.$$
It follows that \vspace{-0.5mm}
$$\left\lVert R(-1,A_0)A_0^{1/2}f\right\rVert^2\leq \lVert f\rVert_
  {\mathbb{H}_{-\frac{1}{2}}}^2.  $$
This, combined with \rfb{maincov}, implies that the first 
norm on the right side of \rfb{II} converges to zero. 

Note that the square of the second norm on the right side of \rfb{II} is
$$ \frac{8\m u^2}{\mu^2\pi} \sum_{k\in\nline}\left|G_\mu(k)
   \right|^2\left|\sum_{l\in\nline} H_\mu(k,l)\right|^2,$$
where $G_\mu(k)$ and $H_\mu(k,l)$ are defined in \rfb{Gmu} and 
\rfb{Bmuformula}, respectively. Different with Step 1 we need here a more 
precise estimate for $G_\mu(k)$ and $H_\mu(k,l)$, such that the double 
series goes to zero as $\mu\to 0$. Besides \rfb{Hmu}, we have the 
following alternative estimate
\begin{equation}\label{Hmu2}
  \sum_{l\in\nline} H_\mu(k,l)=\sum_{l\in\nline}\frac{\mu}{l^2+
  (\sqrt{\mu}k)^2}\leq\sum_{l\leq\sqrt{\mu}k}\frac{1}{ k^2}+
  \sum_{l\geq\sqrt{\mu}k}\frac{\mu}{l^2}\leq 2\frac{\sqrt{\mu}}{k},
\end{equation}
where we used the fact $\sum_{k\geq a}\frac{1}{k^2}\leq\frac{1}{a}$.
Recalling the definition of $G_\mu(k)$ and the function 
$h$ introduced in the proof of Lemma \ref{u=0}, we have 
$$|G_\mu(k)|=\left|\frac{k(1-h_{\mu,k})(1-k^2
	h_{\mu,k})}{(1+k^2h^2_{\mu,k})(1+k^2)}
  \right|,$$
where $h_{\mu,k}=\left(h(\sqrt{\mu}k)\right)^{1/2}$. If $\sqrt{\mu}k
\leq\delta<1$, we still use the estimate \rfb{Gmu1} $|G_\mu(k)|\leq C\sqrt{\mu}$. 
It is not difficult to see that in this case we further have $|G_\mu(k)|\leq 
C\frac{\mu^{1/4}}{k^{1/2}}$. If $\sqrt{\mu}k>\delta$, there exists $c>0$ 
such that $\tanh(\sqrt{\mu}k)\geq c$, which yields that 
$$ 1>h_{\mu,k}\geq \frac{ C }{\mu^{1/4}k^{1/2}}.$$
It follows that we have 
\begin{equation}\label{Gmu2}
 |G_\mu(k)|\m\leq\m \frac{k\left|1-k^2 h_{\mu,k} \right|}{k^4h^2_{\mu,k}}
 \leq C\frac{\mu^{1/4}}{k^{1/2}}. 
\end{equation}
If $k^2h_{\mu,k}\geq1$, \rfb{Gmu2} is a direct consequence. Otherwise, we still 
have
$$ |G_\mu(k)|\leq \frac{1}{k^3\m h_{\mu,k}^2}\leq C \frac{\mu^{1/2}}{k^2}\leq 
   C\frac{\mu^{1/4}}{k^{1/2}},$$
since $\mu<1$. We thus conclude that \rfb{Gmu2} holds for every $\mu\in(0,1)$ 
and $k\in\nline$. Putting together \rfb{Gmu1}, \rfb{Hmu}, \rfb{Hmu2} and \rfb{Gmu2} 
we thus arrive at 
$$ \frac{1}{\mu^2}\sum_{k\in\nline}\left|G_\mu(k)
   \right|^2\left|\sum_{l\in\nline} H_\mu(k,l)\right|^2\leq C \mu^{1/4}.$$

Therefore, the proof of Lemma \ref{A-B} is completed.
\end{proof}

%%%%%%%%%%++++++++++%%%%%%%%%%++++++++++%%%%%%%%%%++++++++++%%%%%%%%%%
\section{Proof of the main result}\label{5}
For any $\omega\in\rline$, we introduce the Hilbert space 
$L^2_\omega[0,\infty):={\rm e}_\omega L^2[0,\infty)$, where 
$({\rm e}_\omega v)(t)={\rm e}^{\omega t} v(t)$ for every $t\geq0$, 
with the norm $\lVert  v\rVert_{L^2_\omega}=\lVert {\rm e}_{-\omega} 
v\rVert_{L^2}$. Similarly, the Hilbert space 
$W^{1,2}_\omega[0,\infty):={\rm e}_\omega W^{1,2}[0,\infty)$ contains 
the elements $({\rm e}_\omega \nu)(t)={\rm e}^{\omega t} \nu(t)$ 
for every $\nu\in W^{1,2}[0,\infty)$, with the norm $\lVert\nu\rVert_{W^{1,2}_\omega}
=\lVert {\rm e}_{-\omega} \nu\rVert_{W^{1,2}}$. For every $\tau>0$, we 
consider the zero extension of the input space $L^2[0,\tau]$ on $(\tau,\infty)$, 
which gives a subspace of $L^2_\omega[0,\infty)$.

Now we are in a position to prove Theorem \ref{result}.

\begin{proof}[Proof of Theorem \ref{result}] 
To present the proof clearly, we divide it into the following three steps.

{\bf Step 1}: {\em The convergence of a scattering semigroup.} For every 
$u\in\Uscr=L_\omega^2[0,\infty)$, according to Lemma \ref{skew-adjoint}, 
we denote by $\tline_\mu=(\tline_{\mu,t})_{t\geq0}$ 
the $C_0$-semigroup generated by the operator $\mathscr{A}_\mu$, and by 
$\tline=(\tline_t)_{t\geq0}$ the $C_0$-semigroup generated by $\mathscr{A}_0$. 
Since $\mathscr{A}_\mu,\m \mathscr{A}_0\in G(1,0)$, then the growth bound 
$\omega(\tline)=0=\omega(\tline_\mu)$. The solutions of the differential 
equations \rfb{newmu-1st} and \rfb{new-1st} are \vspace{-0.5mm}
$$ w_\mu=\tline_{\mu,t}\m w_{\mu,0}+\Phi_{\mu,t}\m u, \vspace{-1mm}$$
and
$$ w=\mathbb{T}_t\m w_0+\Phi_t\m u, \vspace{-1mm}$$
where the initial data $w_{\mu,0}$ and $w_0$ are introduced in \rfb{initial}. 
Note that it is not difficult to check that $\mathscr{B}_0$ is an admissible 
control operator, then the controllability map $\Phi_t$ defined by \vspace{-1mm}
$$ \Phi_t u=\int_0^t \tline_{t-\sigma}\m\mathscr{B}_0 u(\sigma)\m 
   \dd\sigma,\vspace{-1mm}$$
is bounded from $\Uscr$ to $X$. Similarly, we have $\Phi_{\mu,t}\in
\Lscr(\Uscr, X)$ since $\mathscr{B}_\mu\in \Lscr(\cline,X)$. To justify 
the limit $\displaystyle\lim_{\mu\to0}w_\mu=w$ in $X$, we first define 
bounded operators $\mathfrak{T}_{\mu,t}$ and $\mathfrak{T}_t$ by 
$$\mathfrak{T}_{\mu,t}=\bbm{ \tline_{\mu,t} & \Phi_{\mu,t}\\ 0 &  S_t},
  \qquad \mathfrak{T}_{t}=\bbm{\tline_{t} & \Phi_{t}\\ 0 &  S_t },$$
where $(S_t)_{t\geq0}$ is the unilateral left shift simigroup on $\Uscr$,
i.e. $S_tu(\xi)=u(\xi+t)$ for every $\xi\geq0$. Then 
$(\mathfrak{T}_{\mu,t})_{t\geq0}$ and $(\mathfrak{T}(t))_{t\geq0}$ 
form $C_0$-semigroups on $X\times\Uscr$, respectively, with the same growth 
bound $\omega>\omega(\tline)=0$ (Such semigroups were used in 
\cite[Section 4.1]{Obs_Book}, \cite{StWe} and \cite{GrCa}). The generators 
of $\mathfrak{T}_{\mu,t}$ and $\mathfrak{T}_{t}$ are 
$$ \Ascr_\mu=\bbm{\mathscr{A}_\mu & \mathscr{B}_\mu\delta_0 \\ 0 & 
	\frac{d}{d\xi}}, \qquad  \Ascr_0=\bbm{\mathscr{A}_0 & \mathscr{B}_0
	\delta_0 \\ 0 & \frac{d}{d\xi} }, $$
where $\delta_0u(\xi)=u(0)$ for every $u\in\Uscr$, and
$$ \Dscr(\Ascr_\mu)=\left\{\bbm{x_0 \\ u_0}\in X\times W^{1,2}_\omega[0,\infty)
   \left|\m \mathscr{A}_\mu x_0+\mathscr{B}_\mu u_0(0)\in X\right.\right\}. $$
Similarly, $\Dscr(\Ascr_0)$ can be defined by using $\mathscr{A}_0$ and 
$\mathscr{B}_0$ in the above set. Here for simplicity we choose 
$\omega\in(0,1)$ such that $1\in\rho(\Ascr_\mu)\cap\rho(\Ascr_0)$.
Setting, for every $\sbm{x \\u}\in X\times \Uscr$,
$$ \bbm{x_{\mu,0} \\ u_0}=R(1:\Ascr_\mu)\bbm{x \\u},\qquad 
   \bbm{x_0 \\ u_0}=R(1:\Ascr_0)\bbm{x \\ u},$$
we have 
\begin{equation}\label{sol-re}
 \left\{ \begin{aligned}
  &x_0-\mathscr{A}_0x_0-\mathscr{B}_0 u_0(0)=x,\\
  &u_0-\frac{du_0}{d\xi}=u.
 \end{aligned} \right.
\end{equation}
The second equation in \rfb{sol-re} admits a unique solution $u_0$ given
via its Laplace transform
$$\hat{u}_0(s)=\frac{\hat{u}(s)-u_0(0)}{1-s}. $$
According to the Paley-Wiener theorem (see, for instance, 
\cite[Theorem 12.4.2]{Obs_Book}), $u_0(0)=\hat{u}(1)$ is the only 
choice such that $\hat{u}_0(s)\in\Hscr^2(\cline_0)$, where $\Hscr^2(\cline_0)$ 
is the Hardy space with $\cline_0=\{s\in\cline|\Re s>0\}$. We obtain from 
\rfb{sol-re} that
$$ x_0=R(1:\mathscr{A}_0)(x+\mathscr{B}_0 u_0(0)). \vspace{-1mm}$$
Similarly, we have \vspace{-1mm} 
$$ x_{\mu,0}=R(1:\mathscr{A}_\mu)(x+\mathscr{B}_\mu u_0(0)).$$
Recalling Lemma \ref{u=0} and Lemma \ref{A-B} we thus conclude that
$\displaystyle\lim_{\mu\to0}x_{\mu,0}= x_0$ in $X$.
It yields that, for every $\sbm{x \\u}\in X\times \Uscr$,
$$ \lim_{\mu\to0}R(1:\Ascr_\mu)\bbm{x \\ u}= R(1:\Ascr_0)\bbm{x \\ u}\qquad 
   \text{in}\quad X\times \Uscr.$$    
By applying the Trotter-Kato theorem (see, for instance, \cite[Chapter 3]{Pazy}), 
it follows that
$$ \lim_{\mu\to0}\mathfrak{T}_{\mu,t} \bbm{x \\ u}= \mathfrak{T}_t
   \bbm{x \\ u}\qquad \text{in}\quad X\times \Uscr,$$   
uniformly with respect to $t$ on compact intervals. Thus we have
$$ \lim_{\mu\to0}(\tline_{\mu,t}\m x+\Phi_{\mu,t}\m u)= \tline_t\m x
   +\Phi_t\m  u\qquad \text{in}\quad X\times \Uscr. $$
In particular (when $u=0$), for every $x\in X$, we have \vspace{-1mm}
\begin{equation}\label{1}
\lim_{\mu\to0}\tline_{\mu,t}x=\tline_t x\qquad  \text{in}\quad X.
\end{equation} 
 
{\bf Step 2}: {We prove that $\displaystyle\lim_{\mu\to0}w_\mu = w$ in $X$.} 
In order to justify this limit, it suffices to prove 
$\displaystyle\lim_{\mu\to0}\tline_{\mu,t}w_{\mu,0}= \tline_t w_0$, 
where $w_{\mu,0}$ and $w_0$ are introduced in \rfb{initial}. 
We first show that for every $\zeta_0\in\mathbb{H}_{\frac{1}{2}}$, we have 
\vspace{-1.5mm}
\begin{equation}\label{last}
 \lim_{\mu\to0}\left(\frac{1}{\mu}A_\mu\right)
  ^{1/2}\zeta_0 = A_0^{1/2}\zeta_0\qquad \text{in}\quad H. \vspace{-1mm}
\end{equation} 
Since $\frac{1}{\mu}A_\mu$ and $A_0$ are diagonalizable we have
$$ \left(\frac{1}{\mu}A_\mu\right)^{1/2}\zeta_0 - A_0^{1/2}\zeta_0
=\sum_{k\in\nline}k\m I_\mu(k)\langle\zeta_0,\varphi_k
\rangle\varphi_k,$$
with 
$$ I_\mu(k)=\left(\frac{\tanh(\sqrt{\mu}k)}{\sqrt{\mu}k}
   \right)^{1/2}-1. $$
Just like estimating $F_\mu(k)$ in Lemma \ref{u=0},
we use the similar argument here and obtain that $|I_\mu(k)|\leq \sqrt{\mu}k$. 
It implies that for every $\bar \zeta\in \mathbb{H}_1$,
$$ \left\lVert\left(\frac{1}{\mu}A_\mu\right)^{1/2}\bar \zeta-A_0^{1/2} 
   \bar \zeta\m\right\rVert^2\leq \mu\m\lVert\bar \zeta
   \rVert_{\mathbb{H}_1}^2.$$
Note that embedding $\mathbb{H}_1\hookrightarrow \mathbb{H}_{\frac{1}{2}}$ is dense 
and continuous, for every $\varepsilon>0$, there exists $\bar{\zeta}
\in\mathbb{H}_1$, such that $\left\lVert\bar \zeta-\zeta_0\right\rVert_{\mathbb{H}
_{\frac{1}{2}}}<\frac{\varepsilon}{3}$. 
Moreover, it is not difficult to check that \vspace{-1mm}
$$ \left(\frac{1}{\mu}A_\mu\right)^{1/2}\in\Lscr\left(\mathbb{H}_{\frac{1}{2}},
   \mathbb{H}\right),\quad  A_0^{1/2}\in \Lscr\left(\mathbb{H}_{\frac{1}{2}},\mathbb{H}
   \right), \vspace{-1mm}$$ 
and their operator norms are uniformly bounded. We thus have the following estimate
\begin{multline*}
\left\lVert\left(\frac{1}{\mu}A_\mu\right)^{1/2}\zeta_0 - A_0^{1/2}\zeta_0 
\right\rVert\leq \left\lVert \left(\frac{1}{\mu}A_\mu\right)^{1/2}\left(
\zeta_0-\bar \zeta\right)\right\rVert\\
+\left\lVert \left(\frac{1}{\mu}A_\mu\right)^{1/2}\bar \zeta-A_0^{1/2}
\bar \zeta \right\rVert+\left\lVert A_0^{1/2}(\bar{\zeta}-\zeta_0)\right\rVert.
\end{multline*}
Hence, for every $\varepsilon>0$, there exists $\mu_0=\left(\frac{\varepsilon}{3}
\right)^2$, such that for every $ \mu<\mu_0$, we have
$$\left\lVert\left(\frac{1}{\mu}A_\mu\right)^{1/2}\zeta_0 - A_0^{1/2}\zeta_0 
\right\rVert\leq C\varepsilon. $$

Now we estimate the following difference by using a trangle inequality, 
\vspace{-0.5mm}
\begin{equation*}
  \left\lVert\tline_{\mu,t}\m w_{\mu,0}-\tline_t\m w_0\right\rVert_X
  \leq \left\lVert\tline_{\mu,t}\m w_{\mu,0}-\tline_{\mu,t}\m w_0\right
  \rVert_X+\left\lVert\tline_{\mu,t}\m w_0-\tline_t\m w_0\right\rVert_X.
\end{equation*}
Note that $\tline_{\mu,t}$ is unitary and \vspace{-1mm}
$$\left\lVert\tline_{\mu,t}\m w_{\mu,0}-\tline_{\mu,t}\m w_0\right\rVert_X
  \leq\lVert\tline_{\mu,t}\rVert_{\Lscr(X)}\left\lVert
  \left(\frac{1}{\mu}A_\mu\right)^{1/2}\zeta_0 - A_0^{1/2}\zeta_0
  \right\rVert, \vspace{-0.5mm}$$
we have $ \displaystyle\lim_{\mu\to0}\tline_{\mu,t}\m (w_{\mu,0}-\m w_0)=0$ 
for every $\zeta_0\in\mathbb{H}_{\frac{1}{2}}$ and $\zeta_1\in\mathbb{H}$.
This, together with \rfb{1}, implies that $\displaystyle\lim_{\mu\to0}
\tline_{\mu,t}\m w_{\mu,0}=\tline_t\m w_0$ in $X$. We thus achieve that 
$\displaystyle\lim_{\mu\to0}w_\mu=w$ in $X$ (that is, for every $t\geq0$,
$\displaystyle\lim_{\mu\to0}\alpha_\mu= \alpha$ and 
$\displaystyle\lim_{\mu\to0}\beta_\mu=\beta$ hold in $H$). \vspace{1mm}

{\bf Step 3}: {We prove that $\displaystyle\lim_{\mu\to0}\zeta_\mu=\zeta$ 
in $C^1([0,\tau];H)\m\cap \m C([0,\tau];\Hscr_{\frac{1}{2}})$.} 
Recalling the definition of $\alpha_\mu$, $\beta_\mu$ in \rfb{alphamu}, 
and $\alpha$, $\beta$ in \rfb{alpha}, we need to translate the 
convergence results in Step 2 in form of the original variables $\zeta_\mu$ 
and $\zeta$. According to Leibniz formula, we obtain from the first 
convergence, $\displaystyle\lim_{\mu\to0}\alpha_\mu= \alpha$ in $H$, that 
\begin{equation*}
\lim_{\mu\to0}\sup_{t\in[0,\tau]}\lVert\zeta_\mu-\zeta\rVert\leq 
\lim_{\mu\to0}\tau \sup_{t\in[0,\tau]}
\lVert \alpha_\mu-\alpha\rVert.
\end{equation*}
We thus arrive at $\displaystyle\lim_{\mu\to0}\zeta_\mu=\zeta$ in 
$C^1([0,\tau];H)$. Moreover, taking the second convergence 
$\displaystyle\lim_{\mu\to0}\beta_\mu=\beta$ into account we further have 
\begin{equation}\label{-1}
  \lim_{\mu\to0}\left[I+\left(\frac{1}{\mu}A_\mu\right)^{1/2}\right] 
  \zeta_\mu = (I+A_0^{1/2})\zeta\qquad \text{in}\quad H.
\end{equation}
%since \vspace{-1mm}
%\begin{equation*}
  %\left\lVert \left[I+\left(\frac{1}{\mu}A_\mu\right)^{1/2}\right]\zeta_\mu-
  %\left(I+A_0^{1/2}\right)\zeta\m\right\rVert
  %\leq \lVert \zeta_\mu-\zeta\rVert +\left\lVert \beta_\mu-\beta\right\rVert.
%\end{equation*}
Notice that, for every $x\geq0$, $\frac{\tanh x}{x}\sim \frac{1}{1+x}$ (that 
is, each function can be controlled by the other one multiplied by a positive 
constant), we obtain that
$$R\left(-1:\left(\frac{1}{\mu}A_\mu\right)^{1/2}\right)\in
\Lscr(H,\Hscr_{\frac{1}{2}}),$$
and its operator norm is uniformly bounded. It follows from \rfb{-1} that 
$$ \lim_{\mu\to0}\left[\zeta_\mu+R\left(-1:\left(\frac{1}{\mu}A_\mu
   \right)^{1/2}\right)
   \left(I+A_0^{1/2}\right)\zeta\right]= 0 \qquad \text{in}\quad 
   \Hscr_{\frac{1}{2}}. $$
Furthermore, we have \vspace{-1mm}
\begin{multline}\label{III}
  \lVert \zeta_\mu-\zeta\rVert_{\Hscr_{\frac{1}{2}}}\leq \left\lVert \zeta_\mu+
  R\left(-1:\left(\frac{1}{\mu}A_\mu\right)^{1/2}\right)\left(I+A_0^{1/2}\right)
  \zeta\m\right\rVert_{\Hscr_{\frac{1}{2}}}\\
  +\left\lVert R\left(-1:\left(\frac{1}{\mu}A_\mu\right)^{1/2}\right)
  \left(I+A_0^{1/2}\right)\zeta+\zeta\m\right\rVert_{\Hscr_{\frac{1}{2}}}. 
  \vspace{-1mm}
\end{multline}
Observing that it remains to prove that the second norm on the right side 
of \rfb{III} converges to zero, we estimate its square, i.e.
$$ \sum_{k\in\nline} k\m |J_\mu(k)|^2|\langle \zeta, 
  \varphi_k\rangle|^2, \vspace{-1mm}$$
where 
$$ J_\mu(k)=\frac{1+k}{1+k\left(\frac{\tanh(\sqrt{\mu}k)}{\sqrt{\mu}k}
	\right)^{1/2}}-1.$$
Still using the function $h_{\mu,k}$ defined in the proof of Lemma 
\ref{A-B}, we have 
$$\left|J_\mu(k)\right|=\frac{k\left(1-h_{\mu,k}\right)}{1+k\m h_{\mu,k}},$$
since $h_{\mu,k}\in(0,1)$. If $\sqrt{\mu}k \leq\delta<1$, there exists 
$c>0$ such that $h_{\mu,k}\geq c$. According to the Taylor expansion of 
$h_{\mu,k}$, we obtain that $1-h_{\mu,k}\leq C\sqrt{\mu}k$, so that 
$$|J_{\mu}(k)|\leq C\sqrt{\mu}k\leq C\mu^{1/4}k^{1/2}.$$
If $\sqrt{\mu}k>\delta$, we have $h_{\mu,k}\geq\frac{C}{\mu^{1/4}k^{1/2}}$, 
which clearly implies that $|J_\mu(k)|\leq C \mu^{1/4}k^{1/2}$. Hence we have
$$ \sum_{k\in\nline} k\m |J_\mu(k)|^2|\langle \zeta, 
   \varphi_k\rangle|^2\leq \sqrt{\mu}\m\lVert\zeta
   \rVert_{\mathbb{H}_{\frac{1}{2}}}^2\leq C \sqrt{\mu}, \vspace{-1mm}$$
where we used $\zeta\in C([0,\infty); \mathbb{H}_{\frac{1}{2}}) $ for every 
$\zeta_0\in \mathbb{H}_{\frac{1}{2}}$ and $\zeta_1\in\mathbb{H} $. 

Therefore, we finish the proof of Theorem \ref{result}.
\end{proof}

\begin{remark}
{\rm The scattering semigroup $(\mathfrak{T}_t)_{t\geq0} $ used in Step 1 
is actually a part of the so-called {\em Lax-Philips semigroup} of index 
$\omega$ introduced in \cite{StWe}.  }
\end{remark}

%%%%%%%%%%++++++++++%%%%%%%%%%++++++++++%%%%%%%%%%++++++++++%%%%%%%%%%
\section{Conclusion}
In this work, we studied the dispersive limit of the linearized water waves 
equation in a rectangle where the fluid domain is actuated by a wave maker 
from the lateral boundary. To achieve this, we first introduced the dimensionless
Dirichlet to Neumann map and the Neumann to Neumann map, so that the governing 
equations of the linear water waves system can be recast to a second-order 
evolution equation in terms of the elevation of the free surface. Secondly,
we used a special change of variables to obtain a new equivalent system which
possesses the same state space with the wave equation. Finally, we employed a 
scattering semigroup to justify the limit from the linear water waves equation 
to the one dimensional wave equation with Neumann boundary control by using the 
famous Trotter-Kato approximation theorem.

We will discussed in the coming paper the strong well-posedness of the linear 
water waves model and the corresponding asymptotic behaviour of the solution.
In a general bounded convex domain, the well-posedness of the linear water waves 
system with a control is still open. At least, the construction of the Dirichlet 
to Neumann operator can be the first step to start this topic.
%%%%%%%%%%++++++++++%%%%%%%%%%++++++++++%%%%%%%%%%++++++++++%%%%%%%%%%

\section*{Acknowledgements}

The author would like to sincerely thank her supervisors Marius Tucsnak 
and David Lannes (from Universit\'e de Bordeaux) for continuous encouragement 
and for valuable advice on this paper. The author also wants to thank Florent 
Noisette (PhD student in Universit\'e de Bordeaux) for helpful discussion
in the proof of the last step of Lemma 4.2.

%\section*{References}
\bibliographystyle{elsarticle-num}
\bibliography{Su_Tucsnak}

\begin{thebibliography}{10}
\expandafter\ifx\csname url\endcsname\relax
  \def\url#1{\texttt{#1}}\fi
\expandafter\ifx\csname urlprefix\endcsname\relax\def\urlprefix{URL }\fi
\expandafter\ifx\csname href\endcsname\relax
  \def\href#1#2{#2} \def\path#1{#1}\fi

\bibitem{Su2020stabilizability}
P.~Su, M.~Tucsnak, G.~Weiss, Stabilizability properties of a linearized water
  waves system, Systems \& Control Letters 139 (2020) 104672.

\bibitem{lannes2013water}
D.~Lannes, The {W}ater {W}aves {P}roblem: {M}athematical {A}nalysis and
  {A}symptotics, Vol. 188, American Math. Soc., Providence, RI, 2013.

\bibitem{lannes2020modeling}
D.~Lannes, Modeling shallow water waves, Nonlinearity 33~(5) (2020) R1.

\bibitem{reid1985boundary}
R.~M. Reid, D.~L. Russell, Boundary control and stability of linear water
  waves, SIAM J. on Control and Optim. 23~(1) (1985) 111--121.

\bibitem{mottelet2000controllability}
S.~Mottelet, Controllability and stabilization of a canal with wave generators,
  SIAM J. on Control and Optim. 38~(3) (2000) 711--735.

\bibitem{alazard2018stabilization}
T.~Alazard, Stabilization of gravity water waves, Journal de Math{\'e}matiques
  Pures et Appliqu{\'e}es 114 (2018) 51--84.

\bibitem{reid1986open}
R.~M. Reid, Open loop control of water waves in an irregular domain, SIAM
  journal on control and optimization 24~(4) (1986) 789--796.

\bibitem{reid1995control}
R.~M. Reid, Control time for gravity-capillary waves on water, SIAM J. on
  Control and Optim. 33~(5) (1995) 1577--1586.

\bibitem{craig2012water}
W.~Craig, D.~Lannes, C.~Sulem, Water waves over a rough bottom in the shallow
  water regime, in: Annales de l'Institut Henri Poincare (C) Non Linear
  Analysis, Vol.~29, Elsevier, 2012, pp. 233--259.

\bibitem{alazard2011water}
T.~Alazard, N.~Burq, C.~Zuily, On the water-wave equations with surface
  tension, Duke Mathematical Journal 158~(3) (2011) 413--499.

\bibitem{alazard2018boundary}
T.~Alazard, Boundary observability of gravity water waves, in: Annales de
  l'Institut Henri Poincar{\'e} C, Analyse non lin{\'e}aire, Vol.~35, Elsevier,
  2018, pp. 751--779.

\bibitem{Weiss1}
G.~Weiss, Admissibility of unbounded control operators, SIAM J. on Control and
  Optim. 27~(3) (1989) 527--545.

\bibitem{Obs_Book}
M.~Tucsnak, G.~Weiss, Observation and {C}ontrol for {O}perator {S}emigroups,
  Birkh\"auser Verlag, Basel, 2009.

\bibitem{whitham2011linear}
G.~B. Whitham, Linear and {N}onlinear {W}aves, Vol.~42, John Wiley \& Sons,
  2011.

\bibitem{LiMa}
J.-L. Lions, E.~Magenes, Non-homogeneous {B}oundary {V}alue {P}roblems and
  {A}pplications. {V}ol. {I}, Springer-Verlag, New York, 1972.

\bibitem{BDDM}
A.~Bensoussan, G.~Da~Prato, M.~C. Delfour, S.~K. Mitter, Representation and
  {C}ontrol of {I}nfinite {D}imensional {S}ystems, 2nd Edition, Systems \&
  Control: Foundations \& Applications, Birkh\"auser, Boston, MA, 2007.

\bibitem{moreira1999exact}
M.~Moreira~Cavalcanti, Exact controllability of the wave equation with
  {N}eumann boundary condition and time-dependent coefficients, in: Annales de
  la Facult{\'e} des sciences de Toulouse: Math{\'e}matiques, Vol.~8, 1999, pp.
  53--89.

\bibitem{lasiecka1989exact}
I.~Lasiecka, R.~Triggiani, Exact controllability of the wave equation with
  {N}eumann boundary control, Applied Mathematics and Optimization 19~(1)
  (1989) 243--290.

\bibitem{Pazy}
A.~Pazy, Semigroups of {L}inear {O}perators and {A}pplications to {P}artial
  {D}ifferential {E}quations, Appl. Math. Sci. 44, Springer-Verlag, New York,
  1983.

\bibitem{StWe}
O.~J. Staffans, G.~Weiss, Transfer functions of regular linear systems, {P}art
  {II}: The system operator and the {L}ax-{P}hillips semigroup, Trans. Amer.
  Math. Society 354 (2002) 3329--3262.

\bibitem{GrCa}
P.~Grabowski, F.~Callier, Admissible observation operators. semigroup criteria
  of admissibility, Integral Equations and Operator Theory 25 (1996) 182--198.

\bibitem{CZ_THE_BOOK}
R.~F. Curtain, H.~Zwart, An {I}ntroduction to {I}nfinite-dimensional {L}inear
  {S}ystems {T}heory, Springer Verlag, New York, 1995.

\bibitem{chandler2015interpolation}
S.~N. Chandler-Wilde, D.~P. Hewett, A.~Moiola, Interpolation of {H}ilbert and
  {S}obolev spaces: quantitative estimates and counterexamples, Mathematika
  61~(2) (2015) 414--443.

\end{thebibliography}

\end{document}